\documentclass [11pt,oneside,a4paper,mathscr]{amsart}

\usepackage{amsfonts, amsmath, amssymb, amsthm,mathtools, stmaryrd}
\usepackage{verbatim}
\usepackage[T1]{fontenc}                
\usepackage[utf8]{inputenc}             
\usepackage[normalem]{ulem}
\usepackage{tikz-cd}
\usepackage{enumerate}
\usepackage{enumitem}
\usepackage{amsrefs}

\usepackage{array}

\usepackage{hyperref}

\theoremstyle{plain}
\newtheorem{thm}{Theorem}[section]
\newtheorem{cor}[thm]{Corollary}

\newtheorem{lem}[thm]{Lemma}
\newtheorem{prop}[thm]{Proposition}

\numberwithin{equation}{section}

\theoremstyle{definition}
\newtheorem{defn}[thm]{Definition}
\newtheorem{example}[thm]{Example}
\newtheorem{rmk}[thm]{Remark}

\theoremstyle{remark}

\newcommand{\BC}{{\mathbb{C}}}

\newcommand{\BP}{{\mathbb{P}}}
\newcommand{\BQ}{{\mathbb{Q}}}
\newcommand{\BR}{{\mathbb{R}}}

\newcommand{\BZ}{{\mathbb{Z}}}

\newcommand{\CC}{{\mathcal C}}

\newcommand{\CE}{{\mathcal E}}
\newcommand{\CF}{{\mathcal F}}
\newcommand{\CG}{{\mathcal G}}

\newcommand{\CI}{{\mathcal I}}

\newcommand{\CK}{{\mathcal K}}
\newcommand{\CL}{{\mathcal L}}
\newcommand{\CM}{{\mathcal M}}

\newcommand{\CO}{{\mathcal O}}
\newcommand{\CP}{{\mathcal P}}

\newcommand{\CS}{{\mathcal S}}

\newcommand{\CX}{{\mathcal X}}

\newcommand{\CZ}{{\mathcal Z}}
\newcommand{\suchthat}{\;\ifnum\currentgrouptype=16 \middle\fi|\;}

\newcommand{\Fg}{{\mathfrak{g}}}

\newcommand{\Fl}{{\mathfrak{l}}}

\newcommand{\Fs}{{\mathfrak{s}}}

\newcommand{\FS}{\mathfrak{S}}

\newcommand{\pt}{{\mathsf{p}}}
\newcommand{\ch}{{\mathrm{ch}}}
\newcommand{\td}{{\mathsf{td}}}

\newcommand{\h}{\mathrm{H}}
\newcommand{\tH}{\tilde{\h}}
\newcommand{\tdd}{\mathsf{td}^{1/2}}
\newcommand{\qi}{\mathsf{q}_{2i}}
\newcommand{\One}{\mathsf{1}}
\newcommand{\Sq}{\mathsf{q}}
\newcommand{\RO}{\mathrm{O}}

\newcommand{\A}{\mathsf{A}}
\newcommand{\Rc}{\mathrm{c}}
\newcommand{\SM}{\mathsf{M}}
\newcommand{\DM}{\mathrm{DMon}}

\DeclareFontFamily{OT1}{rsfs}{}
\DeclareFontShape{OT1}{rsfs}{n}{it}{<-> rsfs10}{}
\DeclareMathAlphabet{\curly}{OT1}{rsfs}{n}{it}

\newcommand\Ext{\operatorname{Ext}}
\newcommand\Hom{\operatorname{Hom}}
\newcommand\End{\operatorname{End}}
\newcommand{\Aut}{\operatorname{Aut}}

\newcommand{\Coh}{\mathrm{Coh}}

\newcommand{\Pic}{\mathop{\rm Pic}\nolimits}

\newcommand{\ST}{\mathsf{ST}}

\newcommand{\Sym}{\mathrm{Sym}}

\newcommand{\Stab}{\mathrm{Stab}}

\newcommand{\id}{\mathrm{id}}

\newcommand{\FM}{\mathsf{FM}}

\newcommand{\Q}{\mathbb{Q}}
\newcommand{\SH}{\mathrm{SH}}
\newcommand{\Kdrein}{\mathrm{K3}^{[n]}}
\newcommand{\tal}{\tilde{\alpha}}
\newcommand{\tde}{\tilde{\delta}}
\newcommand{\tbb}{\tilde{b}}
\newcommand{\sgn}{\mathrm{sgn}}

\newcommand{\Ker}{\mathrm{Ker}}

\newcommand{\Db}{\mathrm{D}^{\textup{b}}}
\newcommand{\Dbn}{\textrm{D}^{\textup{b}}_{\FS_n}}

\DeclareFontFamily{U}{mathc}{}
\DeclareFontShape{U}{mathc}{m}{it}%
{<->s*[1.03] mathc10}{}

\DeclareMathAlphabet{\mathscr}{U}{mathc}{m}{it}

\DeclareMathOperator{\sExtA}{\mathscr{Ext}}
\begin{document}
	\title[]{DERIVED CATEGORIES OF hyper-Kähler MANIFOLDS: Extended Mukai vector and integral structure}
\author[Thorsten Beckmann]{Thorsten Beckmann}
\newcommand{\Addresses}{{
		\bigskip
		\footnotesize
		
		\textsc{Max--Planck--Institut für Mathematik,
			Vivatsgasse 7, 53111 Bonn, Germany.}\par\nopagebreak
		\textit{E-mail address}: \texttt{beckmann@math.uni-bonn.de}
}}
\maketitle
\begin{abstract}
	We introduce a linearised form of the square root of the Todd class inside the Verbitsky component of a hyper-Kähler manifold using the extended Mukai lattice. This enables us to define a Mukai vector for certain objects in the derived category taking values inside the extended Mukai lattice which is functorial for derived equivalences. As applications, we obtain a structure theorem for derived equivalences between hyper-Kähler manifolds as well as an integral lattice associated to the derived category of hyper-Kähler manifolds deformation equivalent to the Hilbert scheme of a K3 surface mimicking the surface case.
\end{abstract}

\section{Introduction}
\label{sec:in}
{\let\thefootnote\relax\footnotetext{The author is supported by the International Max--Planck Research School on Moduli Spaces of the Max--Planck Society.\\
2020 \textit{Mathematics Subject Classification.} 14J42, 18G80, 14J60, 14C17.\\
\textit{Key words and phrases.} Hyper-Kähler manifolds, derived categories, moduli spaces, Fourier--Mukai partners, Mukai vector.}}
\subsection{Background: Derived categories of K3 surfaces}
The study of derived categories of smooth projective varieties goes back to the works of Mukai \cite{MukaiModK3, MukaiPoincareFourier}. Over the years derived categories and equivalences between them have attracted great attention and culminated in many results, see for example \cite{BridgelandK3, BridgelandSt, BKR, OrlovEquiFM, BondalOrlovReconstruction, HLDKConj}.

Let $X$ and $Y$ be smooth projective varieties. We denote by $\Db(X) \coloneqq \Db(\Coh(X))$ the bounded derived category of coherent sheaves on $X$. Orlov \cite{OrlovEquiFM} showed that any derived equivalence $\Phi \colon \Db(X) \cong \Db(Y)$ is isomorphic to a Fourier--Mukai functor $\FM_\CE$ with Fourier--Mukai kernel $\CE \in \Db(X \times Y)$. In particular, using the Mukai vector
\[
v = \ch(\_) \tdd \colon \Db(X) \to \h^\ast(X,\BQ)
\]
with $\tdd$ the formal square root of the Todd class $\td \in \h^\ast(X,\BQ)$ 
the Fourier--Mukai functor $\FM_\CE$ induces an isomorphism 
\[
\Phi^{\h} = \FM_\CE^{\h} \colon \h^\ast(X,\BQ) \cong \h^\ast(Y,\BQ). 
\]

Let us specialize the above to the case of K3 surfaces $S$. Any smooth variety $Y$ which is derived equivalent to $S$ is again a K3 surface \cite{BridgelandMaciociaSurfaces}. The integral cohomology groups $\h^\ast(S,\BZ)$ are equipped with the \textit{Mukai pairing} $\tbb$ which is equal to the intersection pairing up to a sign $\tbb(\One,\pt) = -1$ for $\One \in \h^0(S,\BZ)$ the fundamental class and $\pt \in \h^4(S,\BZ)$ the point class. Moreover, the lattice $\h^\ast(S,\BZ)$ carries a weight-two Hodge structure inherited from $\h^\ast(S,\BZ)$. 
As alluded to above for a derived equivalence $\Phi \colon \Db(S) \cong \Db(S')$ between two K3 surfaces the following diagram commutes
\begin{equation}
\label{diag:intro_commut_K3}
\begin{tikzcd}
\Db(S)\ar{r}{\Phi}\ar{d}{v} & \Db(S')\ar{d}{v}\\
\h^{\ast}(S,\BZ) \ar{r}{\Phi^{\h}} & \h^{\ast}(S',\BZ).
\end{tikzcd}
\end{equation}
Mukai has shown that the morphism $\Phi^{\h}$ associated to $\Phi$ is a Hodge isometry \cite{MukaiModK3}. Furthermore, the lattice $\h^{\ast}(S,\BZ)$ together with its Hodge structure determines the derived category completely. That is, two 
K3 surfaces $S$ and $S'$ are derived equivalent if and only if $\h^{\ast}(S,\BZ)$ and $\h^{\ast}(S',\BZ)$ are Hodge isometric \cite{OrlovEquiFM}. 

In particular, many properties of the derived category of a K3 surface $S$ are encoded by the lattice $\h^\ast(S,\BZ)$ of rank $b_2(S)+2$ together with its Hodge structure. For example, this can be used to show that the number of Fourier--Mukai partners of $S$, that is the number of non-isomorphic K3 surfaces $S'$ which are derived equivalent to $S$, is finite. 

In addition, the group of auto-equivalences $\Aut(\Db(S))$ of K3 surfaces admits a representation
\[
\rho^{\h} \colon \Aut(\Db(S)) \to \h^\ast(S,\BZ).
\]
The group $\Aut(\Db(S))$ contains elements such as spherical twists $\ST_\CE$ along spherical objects $\CE \in \Db(S)$. These are symmetries which become only visible in the derived category. The image of $\rho^{\h}$ has been computed to be $\Aut^+(\h^\ast(X,\BZ))$, the group of Hodge isometries with real spinor norm one \cite{HLOYAutoK3, HMSK3Orientation, MukaiModK3, OrlovEquiFM}. A conjecture describing the kernel of $\rho^{\h}$ has been put forward by Bridgeland \cite{BridgelandK3} and has been proven for K3 surfaces with Picard rank one \cite[Thm.\ 1.3]{BayerBridgeland}.

The main goal of this paper is to find suitable analogues for the above results for the higher-dimensional analogues of K3 surfaces, that is hyper-Kähler manifolds $X$. For example, we want to study their derived categories and equivalences between them by means of an integral lattice of rank $b_2(X)+2$.
\subsection{Hyper-Kähler manifolds}
Let $X$ be a compact irreducible hyper-Kähler manifold of dimension $2n$, that is a compact simply connected Kähler manifold whose space of holomorphic two-forms is spanned by a non-degenerate symplectic form. We briefly recall properties of $X$ needed to state our results, see Section~\ref{sec:prelim} for a more thorough recollection. 

The second cohomology $\h^2(X,\BQ)$ of $X$ is endowed with a quadratic form called the Beauville--Bogomolov--Fujiki (BBF) form $b$. The Verbitsky component \[\SH(X,\BQ)\subset \h^{\ast}(X,\BQ)\] of $X$ is the 
subalgebra generated by all cohomology classes of degree two. It inherits a bilinear form
\[
b_{\textup{SH}}(\omega_1 \cdots \omega_m, \mu_1 \cdots \mu_{2n-m})=(-1)^m \int_X \omega_1 \cdots \omega_m \mu_1 \cdots \mu_{2n-m}
\]
called \textit{Mukai pairing} from the intersection pairing, where $\omega_i,\mu_j \in \h^2(X,\BQ)$. The extended rational Mukai lattice of $X$ is the graded vector space
\[
\tH(X,\BQ)\coloneqq \BQ \alpha \oplus \h^2(X,\BQ) \oplus  \BQ \beta.
\]
This space is endowed with a bilinear form $\tbb$ as well as a Hodge structure which both restrict to the BBF form and the Hodge structure on $\h^2(X,\BQ)$, respectively. The classes $\alpha$ and $\beta$ are of Hodge type, orthogonal to $\h^2(X,\BQ)$ and satisfy $\tbb(\alpha, \alpha) = \tbb(\beta,\beta)=0$ as well as $\tbb(\alpha,\beta) = -1$ such that $\tH(X,\BQ)$ resembles $\h(S,\BQ)$ for $S$ a K3 surface, see Section~\ref{subsec_verb_comp_prelim} for more details. There exists an embedding
\begin{equation*}
\begin{tikzcd}
\psi \colon \SH(X,\BQ) \arrow[r,hook]&  \Sym^n(\tH(X,\BQ)) \arrow[out=150,in=30]{l}[swap]{T}
\end{tikzcd}
\end{equation*}
of quadratic spaces and $T$ denotes the the orthogonal projection onto the subspace $\SH(X,\BQ)$. The morphism $\psi$ realizes the Verbitsky component as an irreducible representation of the Looijenga--Lunts--Verbitsky (LLV) algebra $\Fg(X)$, see \cite{TaelmanDerHKLL,LooijengaLunts,VerbitskyCohomologyHK} and Section~\ref{subsec_verb_comp_prelim}. 

Let us turn now to derived categories of hyper-Kähler manifolds $X$. Until recently not much has been known about $\Db(X)$.  Huybrechts--Nieper-Wißkirchen have shown that any Fourier--Mukai partner of $X$ is again a hyper-Kähler manifold \cite[Thm.\ 0.4]{HuybrechtsNieperWisskirchen}. 
Taelman in \cite{TaelmanDerHKLL} has refined the study of the derived category of $X$. He showed that a derived equivalence $\Phi\colon \Db(X)\cong \Db(Y)$ 
between hyper-Kähler manifolds restricts to a Hodge isometry 
\[
\Phi^{\SH}\colon \SH(X,\BQ)\cong \SH(Y,\BQ)
\]
which is functorially induced by a Hodge isometry
\begin{equation}
\label{eq:intro_tH_isometry}
\Phi^{\tH}\colon \tH(X,\BQ)\cong \tH(Y,\BQ),
\end{equation}
see \cite[Sec.\ 4]{TaelmanDerHKLL} or Section~\ref{subsec:prelim_der_equi_tael}. This can be used to show that for derived equivalent hyper-Kähler manifolds $X$ and $Y$ there is an isomorphism
\[
\h^i(X,\BQ) \cong \h^i(Y,\BQ)
\]
of $\BQ$-Hodge structures for all $i$ \cite[Thm.\ D]{TaelmanDerHKLL}. 
\subsection{Extended Mukai vector}
The starting point of this paper is the following observation.
\theoremstyle{plain}
\newtheorem*{prop:toddinab}{Proposition~\ref{prop:todd_in_alphabeta}}
\begin{prop:toddinab}
	Let $X$ be a hyper-Kähler manifold of dimension $2n$. Then
	\[
	\overline{\tdd}=T\left(\frac{(\alpha + r_X\beta)^n}{n!}\right) \in \SH(X,\BQ).
	\]
\end{prop:toddinab}
Here, we decompose
\[
\h^{\ast}(X,\BQ) = \SH(X,\BQ) \oplus \SH(X,\BQ)^{\perp}
\]
orthogonally with respect to the intersection product on cohomology and $\overline{(\_)}$ denotes the projection onto the subspace $\SH(X,\BQ)$. The number $r_X\in \BQ$ is an explicit constant depending on $n$, the second Chern class 
$\Rc_2(X)$, and the Fujiki constant $c_X$ of $X$, see \eqref{eq:defn_r}. In particular, it only depends on the deformation type of $X$. For K3 surfaces the proposition reads \[
\tdd=T(\alpha + \beta)=\One + \pt \in \h^{\ast}(X,\BZ)
\]
and for $X$ of $\Kdrein$-type we have
\[
\overline{\tdd}=T\left( \frac{(\alpha + \frac{n+3}{4}\beta)^n}{n!} \right) \in \SH(X,\BQ).
\]

This enables us to define in Section~\ref{sec:vector_line_bundles} an \textit{extended Mukai vector}
\[
\tilde{v}(\CE)\in \tH(X,\BQ)
\]
for certain objects $\CE\in \Db(X)$. It relates to the classical Mukai vector $v(\CE)\in \h^{\ast}(X,\BQ)$ via
\[
T (\tilde{v}(\CE)^n) =c  \overline{v(\CE)}\in \SH(X,\BQ)
\]
for  $c\in \BQ$\footnote{For all known deformation types of hyper-Kähler manifolds we actually have $c\in \BZ$.} and this property characterizes the line spanned by $\tilde{v}(\CE)$ in $\tH(X,\BQ)$. For example, to 
a line bundle $\CL\in \Pic(X)$ with first Chern class $\Rc_1(\CL)=\lambda$ we assign
\[
\tilde{v}(\CL)=\alpha + \lambda + \left(r_X + \frac{b(\lambda,\lambda)}{2}\right) \beta \in \tH(X,\BQ).
\]
Moreover, the formation of the extended Mukai vector is functorial\footnote{The (at first sight surprising) possible extra sign comes from a sign convention in \cite[Thm.\ 4.9]{TaelmanDerHKLL} for certain hyper-Kähler manifolds. For 
all applications this issue can be ignored.} for derived equivalences, i.e.\  for a derived equivalence $\Phi$ the extended Mukai vector of $\Phi(\CE)$ equals $\pm\Phi^{\tH}(\tilde{v}(\CE))$. For the details and precise 
definitions we refer to Section~\ref{sec:vector_line_bundles}. 
\subsection{Derived equivalences of hyper-Kähler manifolds}
Consider a derived equivalence $\Phi \colon \Db(X)\cong \Db(Y)$ between projective hyper-Kähler manifolds $X$ and $Y$. 
Associated to it we have the induced Hodge isometry
\[
\Phi^{\tH}\colon \tH(X,\BQ)\cong \tH(Y,\BQ).
\]
It is a priori very hard to calculate this isometry for a given derived equivalence. However, the above defined extended Mukai vector allows us now to easily compute $\Phi^{\tH}$ for most known examples of derived equivalences between hyper-Kähler manifolds. 
We demonstrate this in Sections~\ref{sec:autoeq_hilb} and \ref{sec:Calculation_exampleequiv_on_extended}. 

Moreover, its properties lead to the following structural result for derived equivalences between hyper-Kähler manifolds. 
\theoremstyle{plain}
\newtheorem*{thm:GeneralStructure}{Theorem~\ref{thm:general_structure}}
\begin{thm:GeneralStructure}
	Let $X$ and $Y$ be deformation-equivalent projective hyper-Kähler manifolds and $\Phi \colon \Db(X) \cong \Db(Y)$ an equivalence with Fourier--Mukai kernel $\CE$. The rank $r$ of $\CE$ is of the form $\frac{a^nn!}{c_X}$ for $a\in \BQ$. If $r=0$, then $\CE$ induces coverings of $X$ and $Y$ with Lagrangian cycles, or there exists a Hodge isometry $\h^2(X,\BZ) \cong \h^2(Y,\BZ)$.
\end{thm:GeneralStructure}
If, for example, $\CE$ is an $X$-flat sheaf on $X \times Y$, then the second statement of the theorem means that the codimension $n$ component of $\textup{supp}(\CE)$ is a flat family of Lagrangian subvarieties of $Y$ which dominates $Y$. 

Note that the number in the above theorem
\[
\frac{a^nn!}{c_X}
\]
must in particular be an integer. For all known examples of hyper-Kähler manifolds $c_X\in \BZ$ and therefore we must already have $a\in \BZ$ 
using Legendre's or de Polignac's formula.

The theorem splits derived equivalences $\Phi = \FM_\CE \colon \Db(X) \cong \Db(Y)$ between hyper-Kähler manifolds into three cases. If the rank of the Fourier--Mukai kernel $\CE$ is non-zero, we are in the first case and the theorem asserts that the possible ranks of $\CE$ are severly restricted. In the second case, the geometries of $X$ and $Y$ are related by a correspondence in the $n$-th Chow group $\A^n(X \times Y)$ which induces coverings of both manifolds by Lagrangian cycles. In the last case the derived equivalence implies the existence of a Hodge isometry $\h^2(X,\BZ) \cong \h^2(Y,\BZ)$. To obtain a geometric interpretation of this conclusion, recall that up to finite index any Hodge isometry is induced from a parallel transport operator \cite[Lem.\ 6.23]{MarkmanSurvey}. The Global Torelli Theorem \cite{VerbitskyTorelli} states that the existence of a Hodge isometry $\h^2(X,\BZ) \cong \h^2(Y,\BZ)$ induced from a parallel transport operator is equivalent to $X$ and $Y$ being birational. 
\subsection{Integral structure}
We now specialize for the rest of the introduction to the case of $\Kdrein$-type hyper-Kähler manifolds $X$, that is hyper-Kähler manifolds which are deformation-equivalent to Hilbert scheme of length $n$ subschemes on a K3 surface. In this case we are able to obtain an integral lattice of rank $b_2(X)+2$ invariant under derived equivalences mimicking the situation for K3 surfaces. 

More explicitly, we define in Section~\ref{sec:lattices} an integral lattice
\[
\Lambda \subset \tH(X,\BQ)
\]
called \textit{$\Kdrein$ lattice} which inherits a Hodge structure $\Lambda_X$ from $X$ through the embedding. As an abstract lattice it is isometric to $\h^2(X,\BZ) \oplus U$ with 
$U$ the hyperbolic plane, but its weight-two Hodge structure differs from 
the one induced from $\h^2(X,\BZ)$ by a B-field twist, see Remark~\ref{rmk:comparison_twisted_HS}. The main difference in the higher-dimensional situation compared to the case of K3 surfaces is that $\h^2(X,\BZ)$ is not unimodular and the B-field twist compensates for the non-trivial discriminant. 

The $\Kdrein$ lattice is a sublattice 
$\Lambda\subset \Lambda_g$ of index two of the lattice $\Lambda_g$ 
generated by all extended Mukai vectors of objects in $\Db(X)$. We refer to Section~\ref{sec:lattices} for a discussion of all the lattices that appear and their relations. 

Our main result now is the following yielding a complete analogue of Mukai's results \cite{MukaiModK3} for derived equivalences of K3 surfaces. 
\theoremstyle{plain}
\newtheorem*{thm:K3nLatticeHIsom}{Theorem~\ref{prop:rel_version_hodge_isom_lattice}}
\begin{thm:K3nLatticeHIsom}
	Let $X$ and $Y$ be projective $\Kdrein$-type hyper-Kähler manifolds 
and $\Phi\colon \Db(X)\cong \Db(Y)$ a derived equivalence. Then $\Phi^{\tH}$ restricts to a Hodge isometry
	\[
	\Phi^{\tH} \colon \Lambda_X\cong \Lambda_Y.
	\]
\end{thm:K3nLatticeHIsom}
Even stronger, the $\Kdrein$ lattice is invariant by the action of all (compositions of) parallel transport operators and derived equivalences acting on the extended Mukai lattice $\tH(X,\BQ)$. The precise statement is Theorem~\ref{thm:derived_mon_grp_K3n}. 

As in the surface case, the existence of a lattice together with a Hodge structure governing properties of the derived category has strong implications. Here is one example. 
\theoremstyle{plain}
\newtheorem*{thm:K3nFiniteFMPartners}{Theorem~\ref{prop:FM_Partners_finite}}
\begin{thm:K3nFiniteFMPartners}
	For a fixed projective $\Kdrein$-type hyper-Kähler manifold $X$ the 
	number of projective $\Kdrein$-type manifolds $Y$ up to isomorphism with $\Db(X)\cong \Db(Y)$ is finite.
\end{thm:K3nFiniteFMPartners}
For all currently known deformation types of hyper-Kähler manifolds a derived equivalence $\Db(X)\cong \Db(Y)$ implies that $X$ and $Y$ must, in fact, be deformation-equivalent. In general, it is not known whether this conclusion remains true for arbitrary hyper-Kähler manifolds.

In \cite[Thm.\ 1.2]{BayMacMMP} it is shown that any hyper-Kähler manifold which is birational to a moduli space of stable objects on a K3 surface $S$ is itself a moduli space of stable objects on $S$. Using the $\Kdrein$ lattice we are able to upgrade this result to derived categories. 
\theoremstyle{plain}
\newtheorem*{cor:DerModuli}{Corollary~\ref{cor:der_equi_moduli_spaces}}
\begin{cor:DerModuli}
	Let $M^S_{\sigma}(v)$ be a smooth moduli space of stable objects on a projective K3 surface $S$ and $X$ a projective $\Kdrein$-type hyper-Kähler manifold such that $\Db(X)\cong \Db(M^S_{\sigma}(v))$. Then $X$ is itself a moduli space of stable objects on $S$.
\end{cor:DerModuli}
The corollary also yields the following.
\theoremstyle{plain}
\newtheorem*{cor:DerModHilb}{Corollary~\ref{cor:der_equi_K3_Hilbn}}
\begin{cor:DerModHilb}
	For two smooth moduli spaces $M^{S}_{\sigma}(v)$ and $M^{S'}_{\sigma'}(v')$ of stable objects on projective K3 surfaces $S$ and $S'$ with $\Db(M^S_{\sigma}(v))\cong \Db(M^{S'}_{\sigma'}(v'))$ we have $\Db(S)\cong \Db(S')$. Furthermore, $S$ and $S'$ are derived equivalent if and only if their Hilbert schemes $S^{[n]}$ and $S'^{[n]}$ are derived equivalent. 
\end{cor:DerModHilb}

Finally, considering a single hyper-Kähler manifold $X$ of $\Kdrein$-type  Theorem~\ref{prop:rel_version_hodge_isom_lattice} implies that the representation 
\[
\rho^{\tH} \colon \Aut(\Db(X)) \to \RO(\tH(X,\BQ))
\]
induced from \eqref{eq:intro_tH_isometry} factors via a representation
\[
\rho^{\tH} \colon \Aut(\Db(X)) \to \Aut(\Lambda_X).
\]
Here, $\Aut(\Lambda_X)$ denotes the group of all Hodge isometries of the $\Kdrein$ lattice $\Lambda_X$.  Specializing to Hilbert schemes $X = S^{[n]}$ of elliptic K3 surfaces $S$ with a section, we are able to give a lower bound on the image of $\rho^{\tH}$. 
\theoremstyle{plain}
\newtheorem*{thm:image_representation_Hilb_U}{Theorem~\ref{prop:image_representation_Hilb_U}}
\begin{thm:image_representation_Hilb_U}
	For the Hilbert scheme $S^{[n]}$ of a K3 surface with $U\subset \textup{NS}(S)$ the image $\textup{Im}(\rho^{\tH})$ of the representation $\rho^{\tH}$ satisfies
	\[
	\hat{\Aut}^+(\Lambda_{S^{[n]}})\subset \textup{Im}(\rho^{\tH}) \subset \Aut(\Lambda_{S^{[n]}}).
	\]
\end{thm:image_representation_Hilb_U}
The group $\hat{\Aut}^+(\Lambda_{S^{[n]}})$ is the group of all Hodge isometries with real spinor norm one which act via $\pm \id$ on the discriminant group. When $2n-2 = 4p^r$ or $2n-2 = 2 p^rq^{s}$ for odd prime numbers $p,q \in \BZ$ and natural numbers $r,s$ the inclusion
\begin{equation*}
	\hat{\Aut}^+(\Lambda_{S^{[n]}}) \subset \Aut^+(\Lambda_{S^{[n]}})
\end{equation*}
is an equality. In these cases Theorem~\ref{prop:image_representation_Hilb_U} determines $\mathrm{Im}(\rho^{\tH})$ up to index two. 

Theorem~\ref{prop:rel_version_hodge_isom_lattice} as well as the existence of the extended Mukai vectors yield several further strong consequences 
for the derived category and derived equivalences of hyper-Kähler manifolds. Instead of reciting all of them here, we invite the reader to directly go to Sections~\ref{sec:vector_line_bundles} and \ref{sec:conclusion_K3nHK_Db}.
\subsection{Related work}
While finishing writing this paper, Eyal Markman informed us that he has also constructed in \cite{MarkmanObs} a Mukai vector with image in the extended Mukai lattice for certain objects in the derived category. His approach uses Hochschild 
(co)homology and obstruction maps and is independent and different from ours. 

While working on this paper we also realised that a broader 
definition of the extended Mukai vector is possible. This has lead to the definition of atomic objects on hyper-Kähler manifolds studied in \cite{BeckmannAtomic}. 
\subsection{Structure of the paper}
In Section~\ref{sec:prelim} we recall results for hyper-Kähler manifolds and their derived categories.

In the first part of this paper we study arbitrary hyper-Kähler manifolds. In Section ~\ref{sec:sqrt_TODD} we prove Proposition~\ref{prop:todd_in_alphabeta} using results from Rozansky--Witten theory. In Section~\ref{sec:vector_line_bundles} we define the extended Mukai vector. There are two different classes of objects for which this can be done and we discuss their properties and give examples. 

In the second part we specialize to hyper-Kähler manifolds deformation-equivalent to the Hilbert scheme of a K3 surface. In Sections \ref{sec:lattices} and \ref{sec:Introcude_DMon} we introduce the lattices that will play a role as well as the derived monodromy group. In Section~\ref{sec:autoeq_hilb} we study derived equivalences of the Hilbert scheme on the 
extended Mukai lattice using the extended Mukai vector. With these preparations we prove in the subsequent section the invariance under derived equivalences of the lattice $\Lambda$. Consequences of the previous results 
will be drawn in Section~\ref{sec:conclusion_K3nHK_Db}. We conclude by demonstrating how known derived equivalences of hyper-Kähler manifolds 
fit into the new set-up.
\subsection{Acknowledgements}
I am indebted to my advisor Daniel Huybrechts for his constant support and encouragement in addition to very helpful comments on a first version of this text. Moreover, I wish to thank Lenny Taelman for many interesting 
discussions on \cite{TaelmanDerHKLL} as well as the content of the present paper 
during an invitation to the Korteweg--de Vries Institute for Mathematics whose hospitality is gratefully acknowledged. I have greatly benefited from discussions with Olivier Benoist, Andreas Mihatsch, Marc Nieper-Wißkirchen, and Georg Oberdieck. 
Furthermore, I want to thank the anonymous referees for helping me to improve the exposition and for spotting a mistake in an earlier version. 
\subsection{Notation}
We will always work over the complex numbers. The derived category $\Db(X)$ of a smooth projective variety $X$ is the bounded derived category of coherent sheaves on $X$. All functors will be implicitly derived. 

A lattice is a free $\BZ$-module of finite rank with an integral (mostly even) quadratic form. We use the notations from \cite[Sec.\ 14]{HuybrechtsK3} and \cite{GHS_Pi1}. We remark that in Section~\ref{sec:computation_upper_bound_DMON_K3n} we use the word lattice as well to denote a full rank discrete subset $W$ inside a finite dimensional rational vector space 
$V$ with a specified embedding $W\hookrightarrow V$. 
It will be clear from the context what is meant. 
\section{Recollections}
\label{sec:prelim}

We recollect facts and results and introduce the notation we will employ throughout the paper. 
\subsection{Hyper-Kähler manifolds and their cohomology}

Let $X$ be a hyper-Kähler manifold of complex dimension $2n$, i.e.\ a simply connected compact Kähler manifold such that $\h^0(X,\Omega_X^2)$ is generated by an everywhere non-degenerate holomorphic two-form. The second cohomology $\h^2(X,\BZ)$ possesses an integral primitive quadratic form $b$ called the \textit{Beauville--Bogomolov--Fujiki (BBF)} form. It is characterized up to sign by the property that there exists a constant $c_X$, called the \textit{Fujiki constant}, that only depends on the deformation type of $X$ such that
\[
\int_X \omega^{2n}=c_X \frac{(2n)!}{2^nn!}b(\omega, \omega)^n
\]
for all $\omega\in \h^2(X,\BZ)$. For the known examples of hyper-Kähler manifolds we have
\begin{equation*}
c_X=\begin{cases} 1&\Kdrein \textup{ or } \mathrm{OG}^{10}\textup{-type,}\\
n+1& \mathrm{Kum}^n \textup{ or } \mathrm{OG}^6\textup{-type.}\end{cases}
\end{equation*}
For the following, see \cite[Cor.\ 23.17]{GroHuyJoyCY}.

\begin{prop}
\label{prop:huybrechts_class_inv_small_defo}
	Let $X$ be a hyper-Kähler manifold of dimension $2n$ and consider a 
class $\mu \in \h^{4p}(X,\BR)$ which is of type $(2p,2p)$ on all small deformations of $X$. Then there exists a constant $C(\mu)\in \BR$ such that
	\[
	\int_X \mu \omega^{2n-2p}=C(\mu)b(\omega,\omega)^{n-p}
	\]
	for all $\omega \in \h^2(X,\BR)$.
\end{prop}

Using Rozansky--Witten theory Nieper-Wißkirchen \cite{NieperWissHRRonIHS}  established results on characteristic classes and Riemann--Roch formulae for hyper-Kähler manifolds. 

\begin{defn}
	For $\omega \in \h^2(X,\BR)$ define its \textit{characteristic value} as
	\[
	\lambda(\omega)\coloneqq \frac{(2n)!12c_X}{2^nn!(2n-1)C(\Rc_2(X))}b(\omega,\omega).
	\]
\end{defn}
Using the Fujiki relations, one can check that the above definition agrees with \cite[Def.\ 17]{NieperWissHRRonIHS}. For the formula of the square root of the Todd class we will need the following result, cf.\ \cite[p.\ 738]{NieperWissHRRonIHS}.
\begin{prop}
	\label{prop:NP_tdd}
	For $X$ a hyper-Kähler manifold of dimension $2n$ and arbitrary $\omega \in \h^2(X,\BR)$ it holds
	\[
	\int_X \tdd \exp(\omega)=(1+\lambda(\omega))^n\int_X\tdd.
	\]
\end{prop}

\subsection{Verbitsky component}
\label{subsec_verb_comp_prelim}
Denote by $(\tilde{\textup{H}}(X,\Q),\tbb)$ the rational quadratic vector 
space defined by \[
\BQ \alpha \oplus \h^2(X,\Q) \oplus \BQ \beta.
\]
The quadratic form $\tilde{b}$ on $\tH(X,\BQ)$ restricts to the BBF form $b$ on $\h^2(X,\BQ)$ and the two classes $\alpha$ and $\beta$ are orthogonal to $\h^2(X,\BQ)$ and satisfy $\tbb(\alpha,\beta)=-1$ as well as $\tbb(\alpha, \alpha)= \tbb(\beta,\beta)=0$. Although it is not integral, we call $\tH(X,\BQ)$ the \textit{extended Mukai lattice} of $X$.

Furthermore, we define on $\tH(X,\BQ)$ a grading by declaring $\alpha$ to 
be of degree zero, $\h^2(X,\BQ)$ remains in degree two and $\beta$ is of degree four. Finally, the extended Mukai lattice is equipped with a weight-two Hodge structure
\begin{align*}
	(\tH(X,\BQ) \otimes_\BQ \BC)^{2,0} &\coloneqq \h^{2,0}(X)\\	
	(\tH(X,\BQ) \otimes_\BQ \BC)^{0,2} &\coloneqq \h^{0,2}(X)\\
	(\tH(X,\BQ) \otimes_\BQ \BC)^{1,1} &\coloneqq \h^{1,1}(X) \oplus \BC \alpha \oplus \BC \beta.
\end{align*}

Let $\textup{SH}(X,\Q)$ be the \textit{Verbitsky component}, i.e.\ the graded subalgebra of $\textup{H}^*(X,\Q)$ generated by $\textup{H}^2(X,\Q)$. Verbitsky \cite{BogomolovVerbitskysResults,VerbitskyCohomologyHK} proved the existence of a graded morphism $\psi\colon \SH(X,\BQ) \to \Sym^n(\tilde{\h}(X,\BQ))$ sitting in a short exact sequence
\[
0 \rightarrow \textup{SH}(X,\Q) \xrightarrow{\psi} \textup{Sym}^n(\tilde{\textup{H}}(X,\Q)) \xrightarrow{\Delta} \textup{Sym}^{n-2}(\tH(X,\Q))\rightarrow 0.
\]
Here, the map $\Delta$ is the Laplacian operator defined on pure tensors via
\[
v_1 \cdots v_n \mapsto \sum_{i<j} \tilde{b}(v_i,v_j) v_1 \cdots  \hat{v_i} \cdots \hat{v_j} \cdots v_n.
\]
The map $\psi$ is uniquely determined (up to scaling) by the condition that it is a morphism of $\mathfrak{g}(X)$-modules, where $\mathfrak{g}(X)$ 
denotes the \textit{Looijenga--Lunts--Verbitsky (LLV)} algebra, see \cite{LooijengaLunts} or \cite{GKLRLLV}. Recall that $\Fg(X)$ is the Lie algebra generated by all $\Fs\Fl_2$-triples $(e_{\omega},h,f_\omega)$ with $e_\omega$ the Lefschetz operator for $\omega \in \h^2(X,\BQ)$ satisfying the Hard Lefschetz property, $h$ the grading operator and $f_\omega$ the dual Lefschetz operator.

The $\Fg(X)$-structure of $\tH(X,\BQ)$ is defined by the conditions $e_\omega(\alpha)=\omega$, $e_\omega(\mu)=b(\omega,\mu)\beta$ and $e_\omega(\beta)=0$ for all classes $\omega, \mu \in \h^2(X,\BQ)$. The $n$-th symmetric power $\Sym^n(\tH(X,\BQ))$ then inherits the structure of a $\Fg(X)$-module by letting $\Fg(X)$ act by derivations. The inclusion realizes $\textup{SH}(X,\Q)$ as an irreducible Lefschetz module \cite{VerbitskyCohomologyHK}. We fix once and for all a choice of $\psi$ by setting $\psi(1)=\alpha^n/n!$. 

Taelman \cite[Sec.\ 3]{TaelmanDerHKLL} showed that the map $\psi$ is an isometry with respect to the \textit{Mukai pairing}
\[
b_{\textup{SH}}(\omega_1 \cdots \omega_m, \mu_1 \cdots \mu_{2n-m})=(-1)^m \int_X \omega_1 \cdots \omega_m \mu_1 \cdots \mu_{2n-m}
\]
on $\textup{SH}(X,\Q)$ and the pairing
\[
b_{[n]}(x_1 \cdots x_n, y_1 \cdots y_n)=(-1)^n c_X \sum_{\sigma \in \mathfrak{S}_n}\prod_{i=1}^n \tbb(x_i,y_{\sigma(i)})
\]
on $\textup{Sym}^n(\tH(X,\BQ))$. Note that our definition of $b_{[n]}$ differs from Taelman's definition by the Fujiki constant. Ours has the advantage that $\psi$ is always an isometry. The orthogonal projection onto the subspace $\SH(X,\BQ)$ will be denoted by 
\[
T \colon \Sym^n(\tH(X,\BQ))\to \SH(X,\BQ).
\]
\begin{rmk}
	\label{rmk:Injectivity_of_T}
	Observe that $\psi$ is surjective in cohomological degrees $0,2,4n-2$ and $4n$. Equivalently, the projection $T$ is injective restricted to these 
degrees.
\end{rmk}
Bogomolov and Verbitsky \cite{BogomolovVerbitskysResults,VerbitskyCohomologyHK} showed that the Verbitsky component can also be described via
\begin{equation}
\label{eq:Bogomolov_Verbitsky_description_SH}
	\SH(X,\BC)\cong\Sym^{\bullet}\h^2(X,	\BC)/\langle \mu^{n+1}\mid b(\mu,\mu)=0 \rangle.
\end{equation}
\subsection{Derived equivalences}
\label{subsec:prelim_der_equi_tael}
The following is \cite[Thm.\ A]{TaelmanDerHKLL}.
\begin{thm}[Taelman]
\label{thm:TaelmanLie_coh_compatible}
	Let $X$ and $Y$ be projective hyper-Kähler manifolds together with an equivalence $\Phi \colon \Db(X)\cong \Db(Y)$. Then $\Phi$ induces a canonical Lie algebra isomorphism
	\[
	\Phi^{\Fg}\colon \Fg(X) \cong \Fg(Y)
	\]
	which is equivariant for the induced isometry $\Phi^{\h}\colon \h^\ast(X,\BQ) \cong \h^\ast(Y,\BQ)$.
\end{thm}
We reproduce some consequences of this result from \cite[Sec.\ 4]{TaelmanDerHKLL} needed below. 
The above theorem implies that given an auto-equivalence $\Phi \in \Aut(\Db(X))$, the induced action on cohomology
\[
\Phi^{\h} \colon \h^\ast(X,\BQ) \cong \h^{\ast}(X,\BQ)
\]
restricts to a Hodge isometry
\[
\Phi^{\SH} \colon \SH(X,\BQ) \cong \SH(X,\BQ)
\]
which is equivariant with respect to $\Phi^{\Fg}$. This yields a representation
\[
\rho^{\SH} \colon \Aut(\Db(X))\to \mathrm{O}(\SH(X,\BQ)).
\] 
Moreover, the above represention $\rho^{\SH}$ factors over a representation
\[
\rho^{\tH} \colon \Aut(\Db(X))\to \mathrm{O}(\tH(X,\BQ))
\]
under the assumption that $n$ is odd, or having $n$ even and $b_2(X)$ odd. Note that all known examples of hyper-Kähler manifolds satisfy one of the two conditions. 

More precisely, for odd $n$ every Hodge isometry $\Phi^{\SH}$ of $\textup{SH}(X,\Q)$ is induced by an isometry of $\textup{Sym}^n(\tilde{\textup{H}}(X,\Q))$ which comes from a unique Hodge isometry $\Phi^{\tH}$ of $\tilde{\textup{H}}(X,\Q)$ \cite[Prop. 4.1]{TaelmanDerHKLL}, i.e.\ the following diagram
\[
\begin{tikzcd}
\SH(X,\BQ)\ar{r}{\Phi^{\SH}}\ar{d}{\psi} &\SH(X,\BQ) \ar{d}{\psi} \\
\Sym^n(\tH(X,\BQ)) \ar{r}{\Sym^n\Phi^{\tH}} & \Sym^n(\tH(X,\BQ))
\end{tikzcd}
\]
commutes. For even $n$ and odd $b_2(X)$, there is an extra sign $\det(\Phi^{\tH})=\epsilon(\Phi^{\tH})\in \{\pm 1\}$ such that
\begin{equation}
\label{diag:main_diag_sec2}
\begin{tikzcd}
\SH(X,\BQ)\ar{r}{\epsilon(\Phi^{\tH})\Phi^{\SH}}\ar{d}{\psi} &\SH(X,\BQ) \ar{d}{\psi} \\
\Sym^n(\tH(X,\BQ)) \ar{r}{\Sym^n\Phi^{\tH}} & \Sym^n(\tH(X,\BQ))
\end{tikzcd}
\end{equation}
commutes. We refer to \cite[Sec.\ 4]{TaelmanDerHKLL} for more details and proofs. 

The process of associating to $\Phi^\SH$ the isometry $\Phi^{\tH}$ is non-trivial. Given an equivalence we cannot say directly how it will act on $\tilde{\textup{H}}(X,\Q)$, e.g.\ there is no obvious cycle associated to 
the kernel of the equivalence. We circumvent this obstacle by using the extended Mukai vector. 
\section{Square root of the Todd class}
\label{sec:sqrt_TODD}
Denote by $\overline{\tdd}$ the projection of the square root of the Todd 
class to the Verbitsky component $\SH(X,\BQ)$. The main goal of this section is to express this class in terms of the extended Mukai lattice. Throughout this section $X$ will be a fixed hyper-Kähler manifold of dimension $2n$ of arbitrary deformation type.

Let us define a number
\begin{equation}
\label{eq:defn_r}
r_X\coloneqq\frac{C(\Rc_2(X))2^{n}n!(2n-1)}{(2n)!24c_X}
\end{equation}
where $C(\Rc_2(X))$ is the constant from Proposition~\ref{prop:huybrechts_class_inv_small_defo} associated to the second Chern class $\Rc_2(X)$. The number relates the BBF form and the characteristic value via
\begin{equation}
\label{eq:number_r_relates_BBF_characteristic}
b(\omega,\omega)=2 r_X \lambda(\omega)
\end{equation}
for all $\omega \in \h^2(X,\BQ)$. 
\begin{lem}
\label{lem:degree_tdd_r}
	The following equality holds
	\[
	\int_X\tdd=c_X\frac{r_X^n}{n!}.
	\]
\end{lem}
\begin{proof}
	Let $\omega \in \h^2(X,\BR)$ be a Kähler class and $t$ a formal variable. Proposition~\ref{prop:NP_tdd} gives
	\begin{equation}
	\label{eq:3.1.1}
	\int_X\tdd\exp(t\omega)=(1+\lambda(t\omega))^n\int_X\tdd.
	\end{equation}
	Both sides are even polynomials in $t$ of degree $2n$. 
	Comparing the coefficient in front of $t^{2n}$ in \eqref{eq:3.1.1} and using
	\[
	\tdd=\One + \textup{terms of higher degree}
	\]
	we obtain
	\[
	c_X\frac{(2n)!}{2^nn!}\frac{b(\omega,\omega)^n}{(2n)!}=\lambda(\omega)^n\int_X\tdd.
	\]
	Solving for $\int_X\tdd$ and employing \eqref{eq:number_r_relates_BBF_characteristic} yields the assertion.
\end{proof}
For the known examples of hyper-Kähler manifolds of dimension $2n$ we have
\[
r_X=
\begin{cases}
	\frac{n+3}{4} & \Kdrein \textup{ or } \mathrm{OG}^{10}\textup{-type,}\\
	\frac{n+1}{4} &\mathrm{Kum}^n \textup{ or } \mathrm{OG}^6\textup{-type.}	
\end{cases}
\]
Lemma~\ref{lem:degree_tdd_r} for $\Kdrein$- and $\mathrm{Kum}^n$-type hyper-Kähler manifolds was also obtained by Sawon \cite{SawonThesis}.
	
Denote for $0 \leq i \leq n$ by $\qi\in \SH^{4i}(X,\BQ)$ the class defined by the property
\[
\int_X \qi\omega^{2n-2i}=c_X\frac{(2n-2i)!}{2^{n-i}(n-i)!}b(\omega,\omega)^{n-i}
\]
for all $\omega \in \h^2(X,\BQ)$.  
\begin{lem}
	\label{lem:monodromy_inv_sh}
	Let $X$ be a hyper-Kähler manifold. Then the subspace of $\textup{SH}^{2i}(X,\Q)$ being of type $(i,i)$ on all small deformations is one-dimensional if $i$ is even and zero 
otherwise. These subspaces are generated by $\qi$. 
\end{lem}
\begin{proof}
	This follows from Proposition~\ref{prop:huybrechts_class_inv_small_defo}. 
\end{proof}
Using Lemma~\ref{lem:monodromy_inv_sh} let us write
\[
\overline{\tdd}=\Sq_0 +a_2 \Sq_2 + \dots + a_{2n-2}  \Sq_{2n-2} + \frac{r_X^n}{n!}\Sq_{2n} \in \SH(X,\BQ)
\]
for $a_{2i}\in \BQ$. We will now determine the remaining coefficients. 
\begin{lem}
\label{prop:coeff_tdd_1_n}
	For $1 \leq i \leq n$ we have
	\[
	a_{2i}=\frac{r_X^i}{i!}.
	\]
\end{lem}
\begin{proof}
	We use again \eqref{eq:3.1.1} and this time compare the coefficients in front of $t^{2n-2i}$. This reads
	\[
	a_{2i}c_X\frac{b(\omega,\omega)^{n-i}}{2^{n-i}(n-i)!}={n \choose n-i}\frac{r_X^ic_X}{2^{n-i}n!}b(\omega,\omega)^{n-i}.\qedhere
	\]
\end{proof}
Recall the isometric embedding $\psi \colon \SH(X,\BQ) \hookrightarrow \Sym^n(\tH(X,\BQ))$ and the orthogonal projection $T \colon \Sym^n(\tH(X,\BQ)) \to \SH(X,\BQ)$. The class $\overline{\tdd}$ has the following expression.
\begin{prop}
\label{prop:todd_in_alphabeta}
Let $X$ be a hyper-Kähler manifold of dimension $2n$. Then
\[
\overline{\tdd}=T\left(\frac{(\alpha + r_X\beta)^n}{n!}\right) \in \SH(X,\BQ).
\]
\end{prop}
If one ignores the orthogonal projection $T$ for the moment, then the proposition says that (the projection of) the square root of the Todd class can be expressed as the $n$-th power of a linear polynomial. Note that the orthogonal projection $T$ really is necessary since $\alpha^{n-i}\beta^i$ is not in the kernel of the Laplacian operator $\Delta$ for $1\leq i \leq n-1$. 

The key step to prove Proposition~\ref{prop:todd_in_alphabeta} is to relate $\alpha^{n-i}\beta^i$ with the classes $\qi$. By definition we have
\[
T(\alpha^n)=n! \One.
\]
In general the connection is given by the following.
\begin{lem}
\label{lem:relation_q_alphabeta}
	For $1\leq i \leq n$ we have
	\[
	T(\alpha^{n-i}\beta^i)= (n-i)!\qi.
	\]
\end{lem}
\begin{proof}
	By definition of the Mukai pairing on $\SH(X,\BQ)$ and the defining property of $\qi$ the assertion of the lemma is equivalent to
	\[
	b_{\SH}(\omega^{2n-2i},T(\alpha^{n-i}\beta^i))=\int_X \omega^{2n-2i}  T(\alpha^{n-i}\beta^i)=c_X \frac{(2n-2i)!}{2^{n-i}}b(\omega,\omega)^{n-i}
	\]
	for all $\omega \in \h^2(X,\BQ)$. The embedding $\psi \colon \SH(X,\BQ) \hookrightarrow \Sym^n(\tH(X,\BQ))$ is an isometry and $T$ is its orthogonal split. The above is therefore equivalent to 
	\[
	b_{[n]}(\psi(\omega^{2n-2i}),\alpha^{n-i}\beta^i)= c_X\frac{(2n-2i)!}{2^{n-i}}b(\omega,\omega)^{n-i}
	\]
	for all $\omega\in \h^2(X,\BQ)$. Since $\psi$ is a morphism of $\mathfrak{g}(X)$-modules, we have 
	\[
	\psi(\omega^{2n-2i})=\psi(e_{\omega}^{2n-2i}(\One))=e_{\omega}^{2n-2i}(\psi(\One)) \in \Sym^n(\tH(X,\BQ))
	\]
	which is a $\BQ$-linear combination of tensors of the form $\alpha^r\omega^s\beta^t$. Only the tensor $\alpha^i\beta^{n-i}$ pairs non-trivially with $\alpha^{n-i}\beta^{i}$. We claim that
	\[
	e_{\omega}^{2n-2i}(\psi(\One))=\frac{(2n-2i)!}{2^{n-i}(n-i)!i!} b(\omega,\omega)^{n-i} \alpha^i\beta^{n-i}+\dots.
	\]
	
	Indeed, let us choose $N \gg n$ and consider the terms $\alpha^r\omega^s\beta^t$ for $r+s+t=N$. We define the order of such a term as $r$, i.e.\ the exponent of $\alpha$. For general $j \geq 0$ the element $e_\omega^j(\psi(\One))$ decomposes
	\[
	e_\omega^j(\psi(\One)) = a^j_1\left(\frac{\alpha^{N-j}\omega^j}{(n-j)!}\right) + a^j_2\left(\frac{\alpha^{N-j+1}\omega^{j-2}\beta}{(n-j+21)!}\right) + \dots + a^j_{ \lfloor{\frac{j}{2}}\rfloor}\left(\frac{\alpha^{N- \lfloor{\frac{j}{2}}\rfloor} \omega^{j-2 \lfloor{\frac{j}{2}}\rfloor} \beta^{ \lfloor{\frac{j}{2}}\rfloor} }{(N- \lfloor{\frac{j}{2}}\rfloor)!}\right)
	\]
	according to the order of the terms that appear. A straight forward proof by induction shows that 
	\[
	a^j_{k+1} = \frac{ (\lfloor{\frac{j}{2}}\rfloor +k)!}{(\lfloor{\frac{j}{2}}\rfloor -k)!k!2^k}
	\]
	which are the coefficients of the Besse polynomials. This then yields the claim. 
	
	Observing that
	\[
	b_{[n]}(\alpha^i\beta^{n-i},\alpha^{n-i}\beta^i)=c_X (n-i)!i!
	\]
	finishes the proof.
\end{proof}
The above lemma helps to understand the two extra classes $\alpha$ and $\beta$ in the extended Mukai lattice $\tH(X,\BQ)$. Rather than trying to identify the classes $\alpha$ and $\beta$ in $\tH(X,\BQ)$ with a single element in $\SH(X,\BQ)$ as in the case of K3 surfaces one should think simultaneously of all the powers $\alpha^{n-i}\beta^i$ as (multiples of) the powers of the BBF-form $\Sq_2^i \in \SH^{4i}(X,\BQ)$. 
\begin{proof}[Proof of Proposition~\ref{prop:todd_in_alphabeta}]
	This follows immediately from Lemma~\ref{lem:degree_tdd_r}, Lemma~\ref{prop:coeff_tdd_1_n} and Lemma~\ref{lem:relation_q_alphabeta}.
\end{proof}
\begin{rmk}
	Proposition~\ref{prop:todd_in_alphabeta} raises the analogous question for $\overline{\td}$. The formulas in the known examples are as follows. If $X$ is of $\Kdrein$ or $\mathrm{OG}^{10}$-type of dimension $2n$, then
	\[
	\overline{\td}=T\left( \frac{(\alpha+2\beta)\cdots (\alpha +(n+1)\beta)}{n!} \right) \in \SH(X,\BQ)
	\]
	and if $X$ is of $\mathrm{Kum}^n$ or $\mathrm{OG}^6$-type of dimension $2n$, then
	\[
	\overline{\td}=T\left( \frac{(\alpha+\beta)\cdots (\alpha +n\beta)}{n!} \right) \in \SH(X,\BQ).
	\]
	These expressions are obtained by an analogous approach as above using this time the known forms of the Riemann--Roch polynomials \cite{EGL,NieperWissHRRonIHS,OrtizRRPolynomials}. There is a formula for $\td$ similar in fashion as the one given in Proposition~\ref{prop:NP_tdd} using Chebyshev polynomials, see \cite[Thm.\ 5.2]{NieperWissHRRonIHS}.
\end{rmk}
\section{Extended Mukai vector}
\label{sec:vector_line_bundles}
The previous section enables us to define a Mukai vector for interesting objects with image in the extended Mukai lattice. We will distinguish two 
cases using the self-intersection of the vector under consideration.

\subsection{Square $-2r_X$}
Let $X$ be a hyper-Kähler manifold of dimension $2n$. 
A line bundle $\CL \in \Pic(X)$ naturally induces an auto-equivalence
\[
\mathsf{M}_{\CL} \coloneqq \_ \otimes \CL \in \Aut(\Db(X)).
\]
Its action $\SM_\CL^{\h}$ on singular cohomology is given by multiplication with the Chern character of $\CL$, i.e.\
\[
\SM_\CL^\h=\_ \cdot \exp(\lambda) \in \End(\h^\ast(X,\BQ))
\]
where $\lambda=\Rc_1(\CL)\in \h^2(X,\BZ)$. Furthermore, denote by $B_\lambda\in \mathrm{O}(\tilde{\textup{H}}(X,\Q))$ the isometry defined by
\[
B_{\lambda}(r\alpha + \mu + s\beta)=r\alpha + \mu +r\lambda +\left(s + b(\lambda,\mu) + r\frac{b(\lambda,\lambda)}{2} \right) \beta. 
\]
As checked in \cite[Prop.\ 3.2]{TaelmanDerHKLL} we have $\SM_\CL^{\tH}=B_{\lambda}$, i.e.\ if we restrict $\SM_\CL^{\h}$ to $\SH(X,\BQ) \subset \h^{\ast}(X,\BQ)$, then it is given by the natural action of $B_\lambda$ on $\Sym^n(\tH(X,\BQ))$ via the diagram
\begin{equation}
	\label{diag:maindiagram}
	\begin{tikzcd}
		\SH(X,\BQ)\ar{r}{\epsilon(\SM_\CL^{\tH})\SM_\CL^{\SH}}\ar{d}{\psi} &\SH(X,\BQ) \ar{d}{\psi} \\
		\Sym^n(\tH(X,\BQ)) \ar{r}{\Sym^nB_\lambda} & \Sym^n(\tH(X,\BQ))
	\end{tikzcd}
\end{equation}
(where we set $\epsilon(\Phi^{\tH})=1$ for all equivalences if $n$ is odd). By definition 
\[
\alpha^n=\psi(n!\One)=\psi(n!\ch(\CO_X))\in \Sym^n(\tH(X,\BQ))
\]
which yields 
\[
\psi(\ch(\CL))=\frac{B_{\lambda}(\alpha)^n}{n!}=
\frac{\left(\alpha + \lambda + \frac{b(\lambda, \lambda)}{2} \beta \right)^n}{n!} \in \Sym^n(\tH(X,\BQ)).
\]

We will upgrade this to the Mukai vector. 
An immediate consequence from Proposition~\ref{prop:todd_in_alphabeta} is 
that for the trivial line bundle $\CO_X \in \Pic(X)$ and its Mukai vector 
$v(\CO_X)$ one has
\[
\overline{v(\CO_X)}=T\left(\frac{(\alpha+r_X\beta)^n}{n!}\right) \in \SH(X,\BQ)
\]
where $\overline{(\_)}$ denotes again the projection from the cohomology $\h^\ast(X,\BQ)$ to the Verbitsky component $\SH(X,\BQ)$. Note that \eqref{diag:main_diag_sec2} also induces a commutative diagram
\begin{equation}
	\label{diag:commut_T_epsilon}
	\begin{tikzcd}
		\SH(X,\BQ)\ar{r}{\epsilon(\Phi^{\tH})\Phi^{\SH}} & \SH(X,\BQ)  \\
		\Sym^n(\tH(X,\BQ)) \ar{r}{\Sym^n\Phi^{\tH}}\ar{u}{T} & \Sym^n(\tH(X,\BQ)) \ar{u}{T}
	\end{tikzcd}
\end{equation}
for all equivalences $\Phi\in \Aut(\Db(X))$. Since $\SM_\CL(\CO_X)= \CL$ and by the compatibility of the cohomological Fourier--Mukai transform $\SM_\CL^{\h}(v(\CO_X))=v(\CL)$ we infer that 
\begin{equation}
	\label{eq:comp_linebundles_vector_SH}
	\overline{v(\CL)}=T\left(\frac{(\alpha+\lambda+(r_X+\frac{b(\lambda,\lambda)}{2})\beta)^n}{n!}\right) \in \SH(X,\BQ)
\end{equation}
for all line bundles $\CL\in \Pic(X)$.

\begin{defn}
\label{defn:extended_Mukai_line_bundles}
	For $\CL\in \Pic(X)$ we define the \textit{extended Mukai vector} of $\CL$ with $\Rc_1(\CL)=\lambda$ as 
	\[
	\tilde{v}(\CL)=\alpha + \lambda + \left( r_X + \frac{b(\lambda,\lambda)}{2} \right)\beta \in \tH(X,\BQ).
	\]
\end{defn}
With this definition, we have
\begin{equation}
\label{eq:compatibility_ext_vector_lb}
\overline{v(\CL)}=T\left( \frac{\tilde{v}(\CL)^n}{n!}\right) \in \SH(X,\BQ).
\end{equation}
Formula \eqref{eq:compatibility_ext_vector_lb} is a helpful tool to deduce properties of $\Phi^{\tH}$ and compute its action on $\tH(X,\BQ)$ for an auto-equivalence $\Phi\in \Aut(\Db(X))$. Here is one example. 

If $n$ is even, then the functoriality of the representation $\rho^{\tH}\colon \Aut(\Db(X))\to \RO(\tH(X,\BQ))$ depends on the determinant of the isometry of $\tH(X,\BQ)$. A useful criterion to calculate $\epsilon(\Phi^{\tH})$ for an auto-equivalence $\Phi \in \Aut(\Db(X))$ is the following. 

\begin{lem}
	\label{lem:determinant_extended_rank}
	Let $X$ be a projective hyper-Kähler manifold with $n$ even and $b_2(X)$ odd. Assume that a line bundle $\CL$ is sent under an auto-equivalence $\Phi$ to an object $\CF$ with positive rank. Then $\epsilon(\Phi^{\tH})=\det(\Phi^{\tH})=1$. 
\end{lem}
\begin{proof}
	The class $\tilde{v}(\CL)^n/n!\in \Sym^n(\tH(X,\BQ))$ maps under $T$ to $\overline{v(\CL)}$. We know that $\Phi^{\h}$ sends $v(\CL)$ to $v(\CF)$. 
Therefore $\overline{v(\CL)}$ is sent to $\overline{v(\CF)}$ under $\Phi^{\SH}$. 	
	Since $n$ is even, the coefficient in front of $\alpha^n$ in the expression $(\Phi^{\tH}(\tilde{v}(\CL)))^n$ must be positive. The commutativity of \eqref{diag:commut_T_epsilon} forces $\epsilon(\Phi^{\tH})=1$. 
\end{proof}
Instead of twists with line bundles, we can also use other auto-equivalences to define an extended Mukai vector for a larger set of objects. 
\begin{defn}
\label{defn:extended_Mukai_general}
Let $X$ be a projective hyper-Kähler manifold of dimension $2n$ with 
$n$ odd and $\CE\in \Db(X)$ an object such that there exists an auto-equivalence $\Phi \in \Aut(\Db(X))$ with $\Phi(\CO_X)\cong\CE$. We define the 
\textit{extended Mukai vector} of $\CE$ as
\[
\tilde{v}(\CE) \coloneqq \Phi^{\tH}(\tilde{v}(\CO_X))\in \tH(X,\BQ).
\]
\end{defn}
This definition does not depend on the chosen equivalence. For such an object $\CE$ we have the equality
\begin{equation}
\label{eq:Muk_vec_vec_objects}
\overline{v(\CE)}=T\left(\frac{\tilde{v}(\CE)^n}{n!}\right). 
\end{equation}
We will say that such objects are in the \textit{$\CO_X$-orbit}. With this terminology, objects in the $\CO_X$-orbit are \textit{cohomologically linearisable} as in \eqref{eq:Muk_vec_vec_objects}, which means that they admit an extended Mukai vector in the extended Mukai lattice, see also Definition~\ref{def:general_extended}. 
\begin{rmk}
	Ideally one would like to give an analogous definition in the case that $n$ is even and $b_2(X)$ is odd. However, one must be cautious since Definition~\ref{defn:extended_Mukai_general} is not well-defined in this case 
and \eqref{eq:Muk_vec_vec_objects} may not serve as a defining property ($v^n=(-v)^n$ for all elements $v\in \tH(X,\BQ)$). The problem is the extra sign discussed in Section~\ref{subsec:prelim_der_equi_tael}. In other 
words, associating to the natural isometry $\Phi^{\SH}$ an isometry $\Phi^{\tH}$ inducing  $\Phi^{\SH}$ as done in Section~\ref{subsec:prelim_der_equi_tael} is not natural and leads to considering sign conventions when defining an extended Mukai vector. We will give an adhoc definition. 
\end{rmk}

Let $X$ be a projective hyper-Kähler manifold of dimension $2n$ with 
$n$ even and $b_2(X)$ odd and choose once and for all a very general Kähler class $\omega \in \h^2(X,\BR)$. Consider an object $\CE \in \Db(X)$ such that there is an equivalence $\Phi\in \Aut(\Db(X))$ satisfying
\[
\Phi(\CO_X)\cong \CE.
\]
If the rank of $\CE$ is strictly positive, then Lemma~\ref{lem:determinant_extended_rank} forces $\epsilon(\Phi^{\tH})=1$ and for negative rank we obtain $\epsilon(\Phi^{\tH})=-1$. This motivates the following. 
\begin{defn}
We say that $\CE$ is \textit{positive} if $\epsilon(\Phi^{\tH})=1$ and \textit{negative} if $\epsilon(\Phi^{\tH})=-1$. 
\end{defn}
This definition is well-defined, i.e.\ it is independent of the chosen equivalence $\Phi$. 
Keeping the above notation, let us denote
\[
v=\Phi^{\tH}(\tilde{v}(\CO_X))=r\alpha + \lambda + s\beta.
\]
We define the \textit{signum $\sgn(v)\in \{\pm 1\}$} of the vector $v$. 
For $r\neq 0$ we set $\sgn(v)\coloneqq \sgn(r)$ as the signum of the number $r$.  If $r=0$, then the self-pairing of $\tilde{v}(\CO_X)$ forces $\lambda\neq 0$ and $c=b(\omega,\lambda)\neq 0$ as the Kähler class $\omega$ was assumed to be very general. We define in this case $\sgn(v)\coloneqq \sgn(c)$. 
\begin{defn}
\label{defn:vector_OX_evenN}
	The \textit{extended Mukai vector} of $\CE$ is
	\[
	\tilde{v}(\CE)\coloneqq \epsilon(\Phi^{\tH})\sgn(v)v.
	\]
\end{defn}
We also say that such objects are in the \textit{$\CO_X$-orbit}. The extended Mukai vector $\tilde{v}(\CE)$ satisfies a version of \eqref{eq:Muk_vec_vec_objects} namely
\begin{equation}
\overline{v(\CE)}=\epsilon(\Phi^{\tH})T\left( \frac{\tilde{v}(\CE)^n}{n!} \right).
\end{equation}
That is, objects in the $\CO_X$-orbit are cohomologically linearisable, but for negative objects we have to add an extra sign.
\begin{rmk}
	The definition agrees with Definition~\ref{defn:extended_Mukai_line_bundles} for line bundles. The motivation for this definition comes from the notion of a positive vector in the theory of moduli spaces of stable sheaves for K3 surfaces as in \cite[Def.\ 0.1]{YoshiokaAbelian}. Moreover, we expect that the choice of a Kähler 
class is not important, i.e.\ the sign of $b(\omega,\lambda)$ is independent of the chosen Kähler class. In all examples that we have calculated for $\Kdrein$-type hyper-Kähler manifolds $X$ the class $\lambda$ is always a multiple of the class Poincaré dual to a line in a projective space $\BP^n\subset X$, therefore bounding the Kähler cone. For all our applications in subsequent chapters the sign choices will not matter.
\end{rmk}
Definition~\ref{defn:vector_OX_evenN} for the case of even $n$ is up to sign compatible with derived equivalences, i.e.\ for $\CE$ and $\CF$ two objects in the $\CO_X$-orbit and a derived equivalence $\Phi\in \Aut(\Db(X))$ with $\Phi(\CE)\cong \CF$ we have
\begin{equation}
\label{eq:functoriality_O_vectorizable_sign}
\Phi^{\tH}(\tilde{v}(\CE))=\pm \tilde{v}(\CF) \in \tH(X,\BQ)
\end{equation}
respectively
\begin{equation}
\label{eq:functoriality_O_vectorizable_sign_npower}
(\Phi^{\tH}(\tilde{v}(\CE)))^n= (\tilde{v}(\CF))^n \in \Sym^n(\tH(X,\BQ)).
\end{equation}

We list some easy properties.
\begin{lem}
\label{lem:prop_vectorizable}
Let $\CE$ be an object in the $\CO_X$-orbit. 
\begin{enumerate}[label={\upshape(\roman*)}]
	\item The object $\CE$ is a $\BP^n$-object.
	\item Its Mukai vector satisfies $ \langle v(\CE),v(\CE)\rangle =n+1$.
	\item Its extended Mukai vector satisfies $\tbb(\tilde{v}(\CE),\tilde{v}(\CE))=-2r_X$.
	\item The rank of $\CE$ is of the form $\pm a^n$ for $a\in \BZ$.
	\item The rank and determinant of $\CE$ determine $\overline{v(\CE)}$ completely.
\end{enumerate}
\end{lem}
\begin{proof}
For the notion of $\BP^n$-object and their properties see \cite{HuybrechtsThomasP}. The pairing in (ii) is the generalized Mukai product introduced in \cite{CaldararuMukaiII}. 

The first four points follow easily from the definitions and the fact that for an equivalence $\Phi\in \Aut(\Db(X))$ the induced isomorphism $\Phi^{\tH}$ of $\tH(X,\BQ)$ is an isometry. 
Let $r=a^n$ be the rank of $\CE$ and $\lambda=\Rc_1(\CE)$ be its determinant. Equation~\eqref{eq:Muk_vec_vec_objects} implies that we only have to determine the extended Mukai vector of $\CE$. Using that the orthogonal projection $T$ is injective in degrees $0$ and $2$ we deduce that up to sign
\[
\tilde{v}(\CE)=a\alpha + \frac{\lambda}{a^{n-1}} + c \beta\in \tH(X,\BQ)
\]
for $c\in \BQ$ uniquely determined by the property $\tbb(\tilde{v}(\CE), \tilde{v}(\CE))=-2r_X$.
\end{proof}
The lemma implies that the Chern classes of such objects are severely restricted.
\begin{rmk}
	For K3 surfaces $S$ the definition of the extended Mukai vector agrees with the usual Mukai vector for line bundles if we identify $\tH(S,\BZ)\coloneqq\h^2(S,\BZ) \oplus U$ with $\h^{\ast}(S,\BZ)$ via $\alpha \mapsto \One$ and $\beta \mapsto \pt$. Note that the Mukai vectors of topological line bundles generate $\h^{\ast}(S,\BZ)$. For certain K3 surfaces (e.g.\ very general projective K3 surfaces of low degree \cite[Rem.\ 6.10]{BayerBridgeland}) we know that all spherical objects are in the orbit of the structure sheaf $\CO_S$ under the action of the group of auto-equivalences.
\end{rmk}

\subsection{Square 0}
There is another class of objects for which one can naturally define an extended Mukai vector. This does not involve Proposition~\ref{prop:todd_in_alphabeta}.

Lemma~\ref{lem:relation_q_alphabeta} yields that the element $\beta\in \tH(X,\BQ)$ has the property
\[
\beta^n=\psi(c_X \pt)\in \SH(X,\BQ).
\]
Since we have for a point $x \in X$
\[
v(k(x))=\ch(k(x))=\pt \in \h^{4n}(X,\BQ)
\]
we obtain the relation
\[
\psi(v(k(x)))=\frac{\beta^n}{c_X} \in \Sym^n(\tH(X,\BQ))
\]
respectively
\begin{equation}
\label{eq:defining_prop_kx_extended_vector}
v(k(x))=T \left( \frac{\beta^n}{c_X} \right) \in \SH(X,\BQ).
\end{equation}
\begin{defn}
	For a point $x\in X$ and the associated skyscraper sheaf $k(x)$ we define its \textit{extended Mukai vector} as
	\[
	\tilde{v}(k(x))\coloneqq \beta.
	\]
\end{defn}
As in the case of objects in the $\CO_X$-orbit this definition can be extended using derived equivalences.
\begin{defn}
	Let $X$ be a projective hyper-Kähler manifold of dimension $2n$ with $n$ odd, $\CE\in \Db(X)$ an object and $\Phi\in \Aut(\Db(X))$ such that 
$\Phi(k(x))\cong \CE$ for some $x\in X$. We define the \textit{extended Mukai vector} of $\CE$ as
	\[
	\tilde{v}(\CE) \coloneqq \Phi^{\tH}(\beta) \in \tH(X,\BQ).
	\]
\end{defn}
We will say that such objects are in the \textit{$k(x)$-orbit}. The analogous relation to \eqref{eq:defining_prop_kx_extended_vector} for objects in the $k(x)$-orbit reads
\begin{equation}
\label{eq:defining_prop_general_kx_extended_vector}
v(\CE)=T\left(  \frac{\tilde{v}(\CE)^n}{c_X} \right) \in \SH(X,\BQ).
\end{equation}
Again in the case $n$ even and $b_2(X)$ odd one has to be more careful. Let $\CE\in \Db(X)$ be such that there exists $\Phi\in \Aut(\Db(X))$ with $\Phi(k(x))\cong \CE$ for some $x\in X$. Let us again write
\[
v=\Phi^{\tH}(\beta)=r\alpha+\lambda+s\beta.
\]
We define again the \textit{signum $\sgn(v)$} of the vector $v$. As before for $r\neq 0$ we set $\sgn(v) \coloneqq \sgn(r)$. In the case $r=0$ and $\lambda\neq 0$ the Hodge Index Theorem asserts that $c=b(\lambda,\omega)\neq 0$ for all Kähler classes $\omega$. We assign $\sgn(v) \coloneqq \sgn(c)$. Finally for $r=\lambda=0$ we define $\sgn(v) \coloneqq \sgn(s)$.
\begin{defn}
	The \textit{extended Mukai vector} of $\CE$ is defined as
	\[
	\tilde{v}(\CE)\coloneqq \epsilon(\Phi^{\tH})\sgn(v)v.
	\]
\end{defn}
We also say that such objects are in the \textit{$k(x)$-orbit}. As above we have
\begin{equation}
	v(\CE)=\epsilon(\Phi^{\tH})T\left( \frac{\tilde{v}(\CE)^n}{c_X} \right)\in \SH(X,\BQ)
\end{equation}
and the formation of the extended Mukai vector is as in \eqref{eq:functoriality_O_vectorizable_sign} and \eqref{eq:functoriality_O_vectorizable_sign_npower} functorial for derived equivalences.
\begin{lem}
	\label{lem:properties_kx_vectorizable}
	Let $\CE$ be an object in the $k(x)$-orbit. 
	\begin{enumerate}[label={\upshape(\roman*)}]
		\item Its Mukai vector $v(\CE)$ lies in $\SH(X,\BQ)$ and satisfies $b_{\SH}( v(\CE),v(\CE)) =0$.
		\item Its extended Mukai vector satisfies $\tbb(\tilde{v}(\CE),\tilde{v}(\CE))=0$.
		\item The rank of $\CE$ is of the form $\pm \frac{a^nn!}{c_X}$ for $a\in \BQ$.
		\item The rank and determinant of $\CE$ determine $v(\CE)$ completely.
		\item If the rank of $\CE$ is zero, then all Chern classes $\Rc_i(\CE)$ 
are isotropic, that is $\sigma|_{\Rc_i(\CE)}=\sigma\Rc_i(\CE)=0 \in \h^{2i+2}(X,\BC)$ with $\sigma\in \h^0(X,\Omega^2_X)$ a symplectic form. 
	\end{enumerate}
\end{lem}
\begin{proof}
	The first three points follow easily from the definition and the fourth point is analogous to Lemma~\ref{lem:prop_vectorizable}.
	
	Suppose that the rank of $\CE$ is zero and write
	\[
	\tilde{v}(\CE)=\lambda + s\beta \in \tH(X,\BQ)
	\]
	with $\lambda \in \h^{1,1}(X,\BQ)$ and $s\in \BQ$. We will assume that $\lambda \neq 0$, the other case being trivial. Since
	\[
	\tbb(\tilde{v}(\CE),\tilde{v}(\CE))=0
	\]
	we infer that $b(\lambda,\lambda)=0$. Equation~\eqref{eq:defining_prop_general_kx_extended_vector} gives
	\begin{equation}
	\label{eq:in_proof_lem_properties_kx_vector}
	v(\CE)=T\left( \frac{(\lambda + s \beta)^n}{c_X} \right) \in \SH(X,\BQ)
	\end{equation}
	which is supported in cohomological degrees ranging from $2n$ to $4n$. Up to a constant, the degree $2n$ component of $v(\CE)$ equals the $n$-th Chern character $\ch_n(\CE)$ which again up to a constant equals the Chern class $\Rc_n(\CE)$. From \eqref{eq:in_proof_lem_properties_kx_vector} we obtain
	\[
	\Rc_n(\CE)=dT(\lambda^n) \in \SH^{2n}(X,\BQ)
	\]
	for some $d\in \BQ$. Since $\psi$ is a morphism of $\Fg(X)$-modules we get
	\[
	\psi(\lambda^n)=\psi(e_{\lambda}^n(\One))=e_{\lambda}^n \left( \frac{\alpha^n}{n!} \right) =\lambda^n \in \Sym^n(\tH(X,\BQ))
	\]
	where the last equality used that $b(\lambda,\lambda)=0$. Therefore $T(\lambda^n)=\lambda^n\in \SH(X,\BQ)$ and so in particular
	\[
	\Rc_n(\CE)=d \lambda^n.
	\]
	The assertion of the lemma is therefore that $\sigma\Rc_n(\CE)=d\sigma 
\lambda^n=0$. We have
	\begin{equation}
	\label{eq:in_proof_lambda_sigma_0}
	b(\lambda,\lambda)=b(\sigma,\sigma)=b(\lambda,\sigma)=0
	\end{equation}
	which shows that
	\begin{equation}
	\label{eq:zasd}
	(\lambda + \sigma)^{n+1} \in \SH^{2n+2}(X,\BC)
	\end{equation}
	must vanish by using \eqref{eq:Bogomolov_Verbitsky_description_SH}. As each summand in \eqref{eq:zasd} lies in a different piece of the Hodge decomposition, we deduce that $\lambda^n\sigma =0$. 
	
	For $k > n$, induction on $k$ shows that $\sigma \Rc_k(\CE) = 0$ if and only if $\sigma v(\CE)_{2k} = 0$
	where $v(\CE)_{2k}$ denotes the cohomological degree $2k$ part of the Mukai vector.
Equation \eqref{eq:in_proof_lem_properties_kx_vector} gives that the degree $2k$ part of $v(\CE)$ equals
	\[
	v(\CE)_{2k}=\frac{s^{k-n}}{c_X}{n\choose k-n}T( \lambda^{2n-k}\beta^{k-n})  \in \SH^{2k}(X,\BQ).
	\]
	To determine the image of $\lambda^{2n-k}\beta^{k-n}$ under the orthogonal projection note that $T$ is as well a morphism of $\Fg(X)$-modules. Similarly to above, one shows
	\[
	\lambda^{2n-k}\beta^{k-n}=e_{\lambda}^{2n-k}\left( \frac{\alpha^{2n-k}\beta^{k-n}}{(2n-k)!} \right) \in \Sym^n(\tH(X,\BQ))
	\]
	using again that $b(\lambda,\lambda)=0$. Lemma~\ref{lem:relation_q_alphabeta} implies that we have
	\[
	T(\alpha^{2n-k}\beta^{k-n})= (2n-k)!\mathsf{q}_{2k-2n}
	\]
	which yields
	\[
	v(\CE)_{2k}=\frac{s^{k-n}}{c_X}{n \choose k-n}\lambda^{2n-k} \mathsf{q}_{2k-2n} \in \SH(X,\BQ).
	\]
	Ignoring constants we have to show that $\lambda^{2n-k} \mathsf{q}_{2k-2n}\sigma\in \SH(X,\BC)$ vanishes which is equivalent to 
	\[
	\lambda^{2n-k} \mathsf{q}_{2k-2n} \sigma \mu^{2n-k-1}=0\in \SH^{4n}(X,\BC)
	\]
	for all $\mu\in \h^2(X,\BQ)$. This follows similar to above using again \eqref{eq:in_proof_lambda_sigma_0} and the polarized version of the Fujiki relations, see Proposition~\ref{prop:huybrechts_class_inv_small_defo}.
\end{proof}
Note that the number from Lemma~\ref{lem:properties_kx_vectorizable} (iii)
\[
\frac{a^nn!}{c_X}
\]
must in particular be an integer. For all known examples of hyper-Kähler manifolds $c_X\in \BZ$ and therefore we must already have $a\in \BZ$ 
using Legendre's or de Polignac's formula.
\begin{rmk}
	\label{rem:explan_n_factorial}
	We want to comment on the denominators appearing in \eqref{eq:Muk_vec_vec_objects} and \eqref{eq:defining_prop_general_kx_extended_vector} and their relation.
	
	The Chern character in degree $4n$ is of the form 
	\[
	\ch_{2n}(\CE)=\frac{1}{(2n)!}(\Rc_1(\CE)^{2n}+\dots).
	\]
	For a line bundle $\CL \in \Pic(X)$ we have by using the Fujiki relation
	\[
	\frac{\Rc_1(\CL)^{2n}}{(2n)!}=c_X\frac{(2n)!b(\lambda,\lambda)^n}{(2n)!n!2^n}=\frac{c_X}{n!}\frac{b(\lambda,\lambda)^n}{2^n}\in \SH^{4n}(X,\BQ).
	\]
	For the known examples of hyper-Kähler manifolds the BBF form is even and the Fujiki constant is integral. In these cases we have
	\begin{equation}
	\label{eq:relation_denominators_vectorizable}
	\frac{\Rc_1(\CL)^{2n}}{(2n)!}\in \frac{c_X}{n!}\BZ
	\end{equation}
	by identifying the class $\pt$ with $1\in \BZ$. The factor $\frac{1}{c_X}$ for objects in the $k(x)$-orbit paired with $\frac{c_X}{n!}$ from \eqref{eq:relation_denominators_vectorizable} gives the factor $\frac{1}{n!}$ 
for objects in the $\CO_X$-orbit. Note that all our computations take place in the Verbitsky component $\SH(X,\BQ)$ which is over $\BQ$ generated by Chern characters of line bundles. 
\end{rmk}
\subsection{Structural result}
The discussion of the two previous subsections enables us to prove the following general structural result for derived equivalences between hyper-Kähler manifolds. 
\begin{thm}
\label{thm:general_structure}
Let $X$ and $Y$ be deformation-equivalent projective hyper-Kähler manifolds and $\Phi \colon \Db(X) \cong \Db(Y)$ an equivalence with Fourier--Mukai kernel $\CE$. The rank $r$ of $\CE$ is of the form $\frac{a^nn!}{c_X}$ for $a\in \BQ$. If $r=0$, then $\CE$ induces a covering of $X$ and $Y$ with Lagrangian cycles, or there exists a Hodge isometry $\h^2(X,\BZ) \cong \h^2(Y,\BZ)$.
\end{thm}
\begin{proof}
	Let us first assume that either $n$ is odd or that $n$ is even and $b_2(X)$ odd. 
	
	We distinguish three cases depending on the image vector
	\[
	v=a\alpha + \lambda + s\beta \in \tH(X,\BQ)
	\]
	of $\beta$ under $\Phi^{\tH}$. For $a \neq 0$ the assertion on the rank follows from Lemma~\ref{lem:properties_kx_vectorizable}. 
	
	In the case that $a=0$, but $\lambda \neq 0$, we consider the the $n$-th Chern class $\Rc_n(\CE)\in \A^n(X\times Y)$ in the Chow ring with rational coefficients. The compatibility of derived equivalences with the induced maps between Chow and cohomology groups shows that for all $x\in X$ and all $y \in Y$ the cycles $\Rc_n(\CE)|_{x\times Y}\in \A^n(Y)$ respectively $\Rc_n(\CE)|_{X\times y}\in \A^n(X)$ are non-zero. Indeed the cohomological degree $2n$ component of $v(\CE_x)$ is equal to $\Rc_n(\CE_x)$ considered in cohomology which by assumption equals $\lambda^n\in \SH(Y,\BQ)$ and similarly for $X$. This shows that viewing the cycle $\Rc_n(\CE)$ as a family of cycles on $Y$ parametrized by points $x\in X$ these cycles cover $Y$ in the sense that for each $y\in Y$ the family of cycles $\Rc_n(\CE)$ restricts non-trivially to the subvariety $X\times y$. Since being isotropic is a cohomological property, the assertion follows now from 
Lemma~\ref{lem:properties_kx_vectorizable} (v). 
	
	Lastly, we assume that $v=s\beta$ for $s\in \BQ$. This assumption implies that the element $\pt\in \h^{4n}(X,\BQ)$ gets sent to $\pm s^n \pt \in \h^{4n}(Y,\BQ)$ under the induced cohomological Fourier--Mukai transform $\Phi^{\h}$. We can view the image of topological $K$-theory under the Mukai vector map $v(K_{\textup{top}}(X))$ as a lattice inside the full cohomology $\h^{\ast}(X,\BQ)$ equipped with the generalized Mukai pairing. We refer to \cite{AddingtonThomas} for a recollection of topological $K$-theory and its relationship to derived categories and Fourier--Mukai transforms. The isomorphism $\Phi^{\h}$ then induces an isometry between the lattices 
$v(K_{\textup{top}}(X))$ and $v(K_{\textup{top}}(Y))$. Since $\pt \in v(K_{\textup{top}}(X))$ is a primitive element, the same must be true for $\Phi^{\h}(\pt)=\pm s^n\pt$. Therefore $s\in \{\pm 1\}$. 
	
	We may assume without loss of generality that $s=1$. Since $\Phi^{\tH}$ is an isometry, we infer that
	\[
	\Phi^{\tH}(\alpha)= \alpha + \lambda + t\beta \in \tH(Y,\BQ).
	\]
	We claim that already $\lambda \in \h^2(Y,\BZ)$. To see this, consider $\tilde{v}(\CO_X)=\alpha + r_X \beta$ and its image
	\[
	\Phi^{\tH}(\tilde{v}(\CO_X))=\alpha + \lambda + (t+r_X)\beta \in \tH(Y,\BQ). 
	\]
	As above the element $v(\CO_X)$ belongs to $v(K_{\textup{top}}(X))$ and therefore $\Phi^{\h}(v(\CO_X))$ must be contained in $v(K_{\textup{top}}(Y))$. Using that the projection of $v(K_{\textup{top}}(Y))$ to its degree 
two component lands inside $\h^2(Y,\BZ)$ and the compatibility \eqref{diag:commut_T_epsilon} we infer that also $\lambda$ belongs to $\h^2(Y,\BZ)$. 
	
	Employing that $\Phi^{\tH}$ is a Hodge isometry we furthermore conclude that $\lambda \in \h^{1,1}(Y,\BZ)$. Hence, there exists a line bundle $\CL\in \Pic(Y)$ with first Chern class $-\lambda$. Changing $\Phi$ by postcomposing it with $\SM_\CL$ we obtain a derived equivalence $\Db(X)\cong \Db(Y)$ still denoted by $\Phi$ which satisfies $\Phi^{\tH}(\alpha)=\alpha$ and $\Phi^{\tH}(\beta)=\beta$. Thus, $\Phi^{\tH}$ restricts to a Hodge isometry
	\[
	\h^2(X,\BQ) \cong \h^2(Y,\BQ).
	\]
	
	It remains to show that this isometry sends $\h^2(X,\BZ)$ to $\h^2(Y,\BZ)$. We employ the same strategy as above. For $\lambda \in \h^2(X,\BZ)$ the vector
	\[
	\tilde{v}(\CL)=\alpha + \lambda + \left( r_X + \frac{b(\lambda,\lambda)}{2} \right) \beta
	\]
	can be viewed as the extended Mukai vector of the (topological) line bundle $\CL$ with first Chern class $\lambda$. The Mukai vector $v(\CL)$ lies inside $v(K_{\textup{top}}(X))$ and $\tilde{v}(\CL)$ is mapped under $\Phi^{\tH}$ to
	\[
	\alpha + \Phi^{\tH}(\lambda) + \left( r_X + \frac{b(\lambda,\lambda)}{2} 
\right)\beta.
	\]
	As before, the compatibility with topological $K$-theory forces $\Phi^{\tH}(\lambda)$ to lie inside $\h^2(Y,\BZ)$. This finishes the proof in the case $n$ odd or $n$ even and $b_2(X)$ odd.
	
	If we now assume that $n$ as well as $b_2(X)$ are odd, then we cannot apply the results from \cite{TaelmanDerHKLL} as explained in Section~\ref{subsec:prelim_der_equi_tael} directly. That is, given a derived equivalence $\Phi \colon \Db(X) \cong \Db(Y)$ the induced cohomological Fourier--Mukai transform $\Phi^{\h} \colon \h^\ast(X,\BQ) \cong \h^\ast(Y,\BQ)$ still restricts to a Hodge isometry
	\[
	\Phi^{\SH} \colon \SH(X,\BQ) \cong \SH(Y,\BQ),
	\]
	but there may not exist a Hodge isometry $\varphi \in \RO(\tH(X,\BQ))$ such that \eqref{diag:main_diag_sec2} commutes. However, inspecting \cite[Prop.\ 4.1]{TaelmanDerHKLL} and its proof we see that there exists a Hodge isometry $\varphi \in \RO(\tH(X,\BQ))$ unique up to sign such that via \eqref{diag:main_diag_sec2} either $\Phi^{\SH}$ or $-\Phi^{\SH}$ agrees with $\varphi^n$. Reinspecting the above proof we see that this sign discrepancy does not affect the arguments and the proof remains valid also in the case $n$ even and $b_2(X)$ even. 
\end{proof}
\subsection{Concluding remarks and further examples}
\label{subsec:Examples_Vectors}
For general sheaves and objects $\CE \in \Db(X)$ we make the following definition.
\begin{defn}
\label{def:general_extended}
	An object $\CE \in \Db(X)$ admits an \textit{extended Mukai vector} $\tilde{v}(\CE)\in \tH(X,\BQ)$ if there exists $c\in \BQ$ such that
	\[
	\overline{v(\CE)}=cT(\tilde{v}(\CE)^n) \in \SH(X,\BQ).
	\]
\end{defn}
With this definition, the vector $\tilde{v}(\CE)$ is not uniquely defined. One rather considers a one-dimensional subspace $V\subset \tH(X,\BQ)$ and demands that $\overline{v(\CE)}$ lies in the one-dimensional subspace 
$T(V^n)$. In the above two series of examples we considered a certain natural choice of $c\in \BQ$ which then enabled us to define the extended Mukai vector as a uniquely determined vector in $\tH(X,\BQ)$. 

Here are two observations how one can generate new examples of objects admitting an extended Mukai vector from known ones:
\begin{itemize}
	\item If $\Phi\colon \Db(X)\cong \Db(Y)$ is an equivalence, then $\CE \in \Db(X)$ admits an extended Mukai vector if and only if $\Phi(\CE)\in \Db(Y)$ admits an extended Mukai vector. 
	\item Let $\pi \colon \CX \to B$ be a smooth and projective morphism with hyper-Kähler manifolds as fibres and $\CE$ on $\CX$ a $B$-flat sheaf or a $B$-perfect complex. For two points $b,b'\in B$ we have that $\CE|_{\CX_{b}}$ admits an extended Mukai vector if and only if $\CE|_{\CX_{b'}}$ admits an extended Mukai vector.
\end{itemize}
One can prove similar results as in Lemmas~\ref{lem:prop_vectorizable} and \ref{lem:properties_kx_vectorizable} for objects $\CE\in \Db(X)$ satisfying Definition~\ref{def:general_extended} for a fixed $c\in \BQ$. We just mention that if $\CE$ has zero rank, then the projections of all Chern classes of $\CE$ to $\SH(X,\BQ)$ are isotropic. To see this one writes $\tilde{v}(\CE)=\lambda+s\beta$ for $\lambda \in \h^2(X,\BQ)$ and $s\in \BQ$ and uses that $e_{\sigma}(\tilde{v}(\CE)) = e_{\bar{\sigma}}(\tilde{v}(\CE))=0$ for $\sigma$ and $\bar{\sigma}$ the (anti-)holomorphic two-form. 

An important class of cohomologically linearisable objects are line bundles and skyscraper sheaves. We want to give further examples. 

\begin{example}
\label{ex:lagr_Pn_vect_obj}
Let $S$ be a projective K3 surface and $\BP^1\cong C\subset S$ a smooth rational curve with class $l=[C]\in \h^2(S,\BZ)$. This yields a Lagrangian projective space $\BP^n\cong C^{[n]}\subset S^{[n]}$ inside the Hilbert scheme of $n$ points. The proof of Proposition~\ref{prop:STonHilbn} will imply that the structure sheaf $\CO_{\BP^n}$ is in the $\CO_{S^{[n]}}$-orbit. If we write $\h^2(S^{[n]},\BZ)=\h^2(S,\BZ) \oplus \BZ \delta$ where $2\delta$ is the class of the exceptional divisor, one has
\[
\tilde{v}(\CO_{\BP^n})=l+\frac{\delta}{2} + \frac{n+1}{2}\beta \in \tH(S^{[n]},\BQ).
\]
In particular, the projection $\overline{[\BP^n]}$ of the class $[\BP^n]$ 
to $\SH^{2n}(S^{[n]},\BQ)$ equals 
\[
\overline{[\BP^n]}=T\left( \frac{(l + \frac{\delta}{2})^n}{n!} \right) \in \SH^{2n}(S^{[n]},\BQ).
\]
This yields a partial answer to a question posed by Bakker \cite[Q.\ 29]{BakkerLagrangian}. For more on this example, we refer to Proposition~\ref{prop:STonHilbn} and Remark~\ref{rem:Pn_in_Hilb}. 
\end{example}
\begin{example}
	For a very general projective K3 surface $S$ of degree $2g-2$ we will study in Section~\ref{subsec:Example_POINCARE_ADM} the case of the moduli space of stable sheaves $M=M^S_H(0,1,d+1-g)$ which admits naturally a Lagrangian fibration $\pi\colon M\to \BP^g=|H|$. The general fibre $A$ is 
a smooth abelian variety and a degree zero line bundle $\CL$ supported on 
$A$ is an example of an object in the $k(x)$-orbit with
	\[
	\tilde{v}(\CL)=f\in \tH(M,\BQ)
	\]
	where $f\in \h^2(M,\BZ)$ is the image of the ample generator of $\Pic(\BP^g)$ under pullback via $\pi$. For $d=0$ the section $\BP^g\subset M$ again yields an object $\CO_{\BP^g}\in \Db(M)$ in the $\CO_M$-orbit.
\end{example}
\begin{example}
\label{ex:4.15}
	To the universal ideal sheaf $\CI$ on $S\times S^{[2]}$ one associates the Fourier--Mukai kernel \cite{AddingtonDerSymHK,MarkmanMehrotraIntTransf}
	\[
	\mathcal{E}^1 \coloneqq \sExtA^1_{\pi_{13}}(\pi_{12}^*(\mathcal{I}), \pi_{23}^*(\mathcal{I})) \in \Coh(S^{[2]} \times S^{[2]})
	\]
	where $\pi_{ij}$ denote the projections from $S^{[2]}\times S \times S^{[2]}$. Consider a point $p\in S^{[2]}$ parametrizing two distinct points $x,y\in S$ and denote by $Z_x$ respectively $Z_y$ the subvarieties of $S^{[2]}$ parametrizing subschemes whose support contains $x$ respectively $y$. The derived equivalence $\FM_{\CE^1}$ sends $k(p)$ to the sheaf $\CE^1_{p\times S^{[2]}}$ which sits in a short exact sequence
	\[
	0 \to \CO(-\delta) \to \CE^1_{p\times S^{[2]}} \to I_{Z_x\cup Z_y}\to 0
	\]
	and is an example of an object in the $k(x)$-orbit with extended Mukai vector
	\[
	\tilde{v}(\CE^1_{p\times S^{[2]}})=\alpha -\frac{\delta}{2}-\frac{1}{4}\beta.
	\]
	For more on this example, see Section~\ref{subsec:examples_dim4}. 
\end{example}
\begin{rmk}
	\label{rmk:NO_vector_on_all_Db}
	Ideally, one would like to define a vector $\tilde{w}$ for all elements $\CE\in \Db(X)$ in a coherent way. By this we mean for example that its formation should factor through the $K$-group $K(X)$ and is compatible with derived equivalences, i.e.\ the following diagram
	\begin{equation*}
	\begin{tikzcd}
	\Db(X)\ar{r}{\Phi}\ar{d}{\tilde{w}} &\Db(Y) \ar{d}{\tilde{w}} \\
	\tH(X,\BQ) \ar{r}{\Phi^{\tH}} & \tH(Y,\BQ)
	\end{tikzcd}
	\end{equation*}
	should commute. However, this is too much to ask for. 
	
	Indeed, consider for example the case $X=S^{[2]}$ of the Hilbert scheme of two points on a projective K3 surface $S$. 
	Using the Koszul resolution, one can check that the structure sheaf of a 
complete intersection of divisors of codimension larger than two must have trivial image under $\tilde{w}$. In particular, all sheaves supported on a zero-dimensional subscheme must have trivial image under $\tilde{w}$. 
Therefore all previously defined objects in the $k(x)$-orbit must map to zero under $\tilde{w}$. Hence, the vector $\tilde{w}$ vanishes for $\CE^1_{p\times S^{[2]}}\otimes \CL$ for all line bundles $\CL \in \Pic(X)$. It 
therefore also has to vanish on all divisors. 
	
	The same argument as above also shows that any vector
	\[
	K^0_{\textup{top}}(S^{[2]})\xrightarrow{\tilde{w}}\tH(S^{[2]},\BQ)
	\]
	compatible with derived equivalences as above must vanish on all topological line bundles.
\end{rmk}
\section{Integral lattices for $\Kdrein$-type hyper-Kähler manifolds}
\label{sec:lattices}
We want to apply the results and definitions from the previous sections. From now on, $X$ will denote a hyper-Kähler manifold of $\Kdrein$-type with $n>1$. 
\subsection{Lattices}
\label{subsec:lattices}
In this section, we want to discuss the (potential) integral lattices inside the extended Mukai lattice that appear and set up notation which will 
be used throughout the rest of the paper. We will fix for $X$ once and for all an isometry
\begin{equation}
\label{eq:choice_of_delta_for_lattive}
\h^2(X,\BZ)\cong \h^2(S^{[n]},\BZ) \cong \h^2(S,\BZ) \oplus \BZ \delta,
\end{equation}
where $S$ is a K3 surface, $S^{[n]}$ is the $n$-th Hilbert scheme with $2\delta$ the class of the exceptional divisor of the Hilbert--Chow morphism and the second isometry is given by \eqref{eq:isometries_hilb} (for $X=S^{[n]}$ we choose the first isometry to be the identity). 

Let us first quickly review the case of a K3 surface $S$. There, the full 
integral cohomology $\h^{\ast}(S,\BZ)=\tH(S,\BZ)$ with the Mukai pairing is governing the derived category \cite{OrlovEquiFM,MukaiModK3}. That is, an equivalence of K3 surfaces yields a Hodge isometry between the full 
integral cohomologies and two K3 surfaces are derived equivalent if and only if their integral cohomologies are Hodge isometric.  
Moreover, the lattice spanned by the Mukai vectors of topological line bundles equals the full integral cohomology. 

In higher dimensions, the situation changes. There are several relevant lattices. 
\begin{defn}
	We define the \textit{integral extended Mukai lattice} as the lattice
	\[
	\tH(X,\BZ) \coloneqq \BZ \alpha \oplus \h^2(X,\BZ) \oplus \BZ \beta \subset \tH(X,\BQ).
	\]
\end{defn}
Several facts suggest that in higher dimensions this is the wrong lattice 
to look at. Firstly, Definition~\ref{defn:extended_Mukai_line_bundles} and Examples~\ref{ex:lagr_Pn_vect_obj} and \ref{ex:4.15} suggest that one should allow certain denominators. Secondly, we will see in Proposition~\ref{prop:STonHilbn} that derived equivalences do not send integral elements to integral elements. 
\begin{defn}
	For $\delta\in \h^2(X,\BZ)$ as above we define the \textit{$\Kdrein$ lattice} as
	\[
	\Lambda \coloneqq B_{-\delta/2}(\tH(X,\BZ))\subset \tH(X,\BQ).
	\]
\end{defn}
This is independent of our fixed choice of $\delta$, i.e.\ for any class $\gamma\in \h^2(X,\BZ)$ of square $2-2n$ and divisibility $2n-2$ one has 
\[
\Lambda=B_{-\gamma/2}(\tH(X,\BZ))\subset \tH(X,\BQ).
\]
In dimension $4$ this lattice was also considered by Taelman \cite[Thm.\ E]{TaelmanDerHKLL}.

We introduce the notation \[
\tal\coloneqq B_{-\delta/2}(\alpha)=\alpha- \frac{\delta}{2}+ \frac{1-n}{4}\beta, \quad  \tde\coloneqq B_{-\delta/2}(\delta)= \delta + (n-1)\beta.
\] 
With it one can define equivalently the $\Kdrein$ lattice as the lattice
\begin{equation}
\label{eq:ab123}
\Lambda = \BZ \tal \oplus \h^2(S,\BZ) \oplus \BZ\tde \oplus \BZ\beta = 
\Lambda_S \oplus \BZ \tde
\end{equation}
where
\[
\Lambda_S=\BZ \tal \oplus \h^2(S,\BZ) \oplus \BZ \beta.
\]
Note that $\tal$ and $\beta$ still generate an integral hyperbolic plane and that the decomposition $\Lambda_S\oplus \BZ \tde$ is orthogonal. 
The integral extended Mukai lattice and the $\Kdrein$ lattice are isometric as abstract lattices and neither is included in the other when seen inside $\tH(X,\BQ)$. 
\begin{defn}
	The \textit{geometric lattice} $\Lambda_g$ is defined as
	\[
	\Lambda_g\coloneqq \Lambda_S \oplus \BZ \frac{\tde}{2}\subset \tH(X,\BQ).
	\]
\end{defn}
Be aware that the quadratic form of $\Lambda_g$ inherited from $\tH(X,\BQ)$ may not be integral. 

We want to motivate this definition. Recall that $r_X=\frac{n+3}{4}$ and let us look at the lattice generated by all extended Mukai vectors of topological line bundles, i.e.\
\[
\Lambda_{LB}\coloneqq \left\langle \left\{\tilde{v}(\lambda)\coloneqq \alpha + \lambda + \left( \frac{n+3}{4} + \frac{b(\lambda,\lambda)}{2} \right)\beta \suchthat \lambda \in \h^2(X,\BZ) \right\} \right\rangle.
\]
Note that one can write the generators equivalently as
\[
\tilde{v}(\lambda)= \tal + \frac{\tde}{2} + \lambda + \left(1+\frac{b(\lambda,\lambda)}{2}\right)\beta.
\]
If one ignores the term $\frac{\tde}{2}$ for a moment, then the expression resembles the Mukai vector on a K3 surface (where $\tdd=\One + \pt$). 
One can check that as an abstract lattice $\Lambda_{LB}$ is isometric to $\tH(X,\BZ)$.

In Section~\ref{subsec:Examples_Vectors} we saw that there are more objects than line bundles and skyscraper sheaves of points for which we can define an extended Mukai vector. We have
\[
\tilde{v}(\CO_{\BP^n})=l+\frac{\tde}{2}+\beta,\quad \tilde{v}(\CE^1_{p\times S^[n]})=\tal.
\]
\begin{lem}
	The geometric lattice $\Lambda_g$ equals the lattice spanned by $\Lambda_{LB}$ as well as all extended Mukai vectors from Section~\ref{subsec:Examples_Vectors}.
\end{lem}
\begin{proof}
	This follows from a straightforward calculation.
\end{proof}
\begin{rmk}
	We expect that for all elements $\CE\in \Db(X)$ for which a (meaningful) 
extended Mukai vector $\tilde{v}(\CE)$ can be defined, one has $\tilde{v}(\CE)\in \Lambda_g$. We will prove in Corollary~\ref{conj:der_inv_geom_lattice} that $\Lambda_g$ is invariant under all parallel transport isometries as well as derived equivalences. 
\end{rmk}
\subsection{Hodge structures}
All the above defined lattices carry a weight-two Hodge structure from their inclusion into $\tH(X,\BQ)$.

\begin{defn}
	For a lattice $\Gamma \subset \tH(X,\BQ)$ we define its \textit{algebraic part} as
	\[
	\Gamma_{\textup{alg}} \coloneqq \Gamma \cap \tH^{1,1}(X,\BC) \subset \tH(X,\BQ)
	\]
	and its \textit{transcendental part} as
	\[
	\Gamma_{\textup{tr}} \coloneqq \Gamma_{\textup{alg}}^{\perp}\cap \Gamma \subset \tH(X,\BQ).
	\]
\end{defn}
With this definition the transcendental part of the integral extended Mukai lattice equals the \textit{transcendental lattice} of the hyper-Kähler manifold $X$, i.e.\ 
\[
\h^2(X,\BZ)_{\textup{tr}} \coloneqq \h^2(X,\BZ)\cap \textup{NS(X)}^{\perp}=\tH(X,\BZ)_{\textup{tr}} \subset \h^2(X,\BZ).
\]
\begin{lem}
\label{lem:transcendental_hodge_structure}
	The transcendental part $\Lambda_{\textup{tr}}$ of the $\Kdrein$ lattice 
$\Lambda$ equals the transcendental lattice of $X$.
\end{lem}
\begin{proof}
	Both inclusions follow from \eqref{eq:ab123}. 
\end{proof}
\begin{rmk}
\label{rmk:comparison_twisted_HS}
	The isometry $B_{-\delta/2}$ yields an isometry between the integral extended Mukai lattice $\tH(X,\BZ)$ and the $\Kdrein$ lattice $\Lambda$, which in general does not respect the Hodge structures. However, if we endow 
$\tH(X,\BZ)$ with the twisted Hodge structure associated to the B-field $\delta/2\in \h^2(X,\BQ)$ as defined in \cite[Def.\ 2.3]{HuybrechtsStellariTwisted}, then $B_{-\delta/2}$ induces a Hodge isometry between $\tH(X,\BZ)$ endowed with the twisted Hodge structure and $\Lambda$ equipped with 
the Hodge structure coming from the embedding $\Lambda \subset \tH(X,\BQ)$. 
	
	To see this consider a symplectic form $\sigma \in \h^2(X,\BC)$. The twisted Hodge structure is determined by the element $\sigma + \frac{1}{2} b(\sigma,\delta)\beta$ and this is sent under $B_{-\delta/2}$ to the symplectic form $\sigma$. The untwisted and the twisted Hodge structure on $\tH(X,\BZ)$ have the same transcendental lattice, whereas in the case of K3 
surfaces the transcendental lattice of a twisted Hodge structure associated to a non-trivial Brauer class is always a proper sublattice of the transcendental lattice of the untwisted Hodge structure \cite[Sec.\ 2]{HuybrechtsStellariTwisted}.
\end{rmk}
\section{Derived Monodromy group}
\label{sec:Introcude_DMon}
Let $X$ be a hyper-Kähler manifold of $\Kdrein$-type and let $X_1$ and $X_2$ be deformations of $X$. By this we mean smooth and proper morphisms $\pi_i\colon \CX_i \to B_i$ for $i\in \{1,2\}$ with $B_i$ connected such that there is one point $0_i\in B_i$ with $\pi_i^{-1}(0_i) \cong X$  and another point $b_i\in B_i$ such that $\pi_i^{-1}(b_i)\cong X_i$. Let $\gamma_i\colon \SH(X,\BQ) \cong \SH(X_i,\BQ)$ be parallel transport isometries obtained from choosing a path between $0$ and $b_i$ in $B_i$. Moreover, consider a Fourier--Mukai equivalence $f\colon \Db(X_1) \cong \Db(X_2)$ and denote by $F=f^{\SH}$ the induced isometry. 
\begin{defn}[Taelman]
\label{defn:derived_mon_grp}
The \textit{derived monodromy group} $\DM(X)$ is the subgroup of $\mathrm{O}(\SH(X,\BQ))$ generated by all isometries of the form
\[
\SH(X,\BQ)\xrightarrow{\gamma_1} \SH(X_1,\BQ)\xrightarrow{F}\SH(X_2,\BQ)\xrightarrow{\gamma_2^{-1}}\SH(X,\BQ).
\] 
\end{defn}
We know from Section~\ref{subsec:prelim_der_equi_tael} that the isometry $F$ is induced from an isometry $f^{\tH}$. Similarly, the isometries $\gamma_i$ are by \cite[Prop.\ 4.1]{TaelmanDerHKLL} induced by unique Hodge isometries $\gamma_i^{\tH} \in \RO(\tH(X,\BQ))$. This implies that the derived monodromy group has an inclusion $\DM(X) \subset \mathrm{O}(\tH(X,\BQ))$. 
Throughout this paper we will consider the elements of $\DM(X)$ always as 
isometries of the extended Mukai lattice.

Let now $S^{[n]}$ be the Hilbert scheme of $n$ points on a projective K3 surface $S$. Bridgeland, King, and Reid \cite{BKR} proved the existence of a derived equivalence
\[
\Db(S^{[n]})\cong \Dbn(S^n)
\]
where the latter is the $\FS_n$-equivariant derived category of the product variety $S^n$. For an introduction and notation regarding equivariant categories we refer to \cite{BeckmannOberdieckNotes}. For our purposes, we will not take the equivalence from \cite{BKR}, but the one considered by Krug in \cite{KrugRemarks}
\[
\Psi \colon \Dbn(S^n) \cong\Db(S^{[n]}),
\]
since it has nice properties for the computations we want to perform. 
Consider a line bundle $\CL$ on $S$. There is a natural line bundle $\CL_n$ on $S^{[n]}$ associated to $\CL$ which satisfies $\Psi((\CL^{\boxtimes 
n},1))=\CL_n$ \cite[Thm.\ 1.1]{KrugRemarks}. This yields the well-known 
isomorphisms
\begin{equation}
\label{eq:isometries_hilb}
\Pic(S^{[n]})\cong \Pic(S) \oplus \BZ \delta, \quad \h^2(S^{[n]},\BZ) \cong \h^2(S,\BZ) \oplus \BZ \delta
\end{equation}
where $2\delta=[E]$ is the class of the exceptional divisor of the Hilbert--Chow morphism. Since $\h^1(\FS_n,\BC^{\ast})=\BZ/2\BZ$, the simple 
object $\CL^{\boxtimes n}\in \Db(S^n)$ possesses another linearisation given by tensoring with the sign-representation. It holds 
\begin{equation}
	\label{eq:lb_different_linearization}
\Psi(\CL^{\boxtimes n},-1)=\CL_n \otimes \CO_{S^{[n]}}(-\delta),
\end{equation}
where $\CO_{S^{[n]}}(-\delta)\in \Pic(S^{[n]})$ is the line bundle with first Chern class $-\delta$ \cite[Rem.\ 3.10]{KrugRemarks}. 

Ploog \cite{PloogEquivFinite}, later generalized by Ploog--Sosna \cite{PloogSosnaCYHK}, observed that there is an injective group homomorphism
\[
\phi_{[n]}\colon \Aut(\Db(S)) \times \BZ/2\BZ \hookrightarrow \Aut(\Dbn(S^n)).
\]
More precisely, Orlov's Theorem \cite{OrlovEquiFM} asserts that every auto-equivalence of $\Db(S)$ is given by a Fourier--Mukai functor with kernel $\CE\in \Db(S \times S)$. The kernel $\CE^{\boxtimes n}$ can be canonically equipped with a $\FS_n$-linearisation, where $\FS_n$ acts diagonally 
on $S^n \times S^n$. The factor $\h^1(\FS_n, \BC^{\ast})=\BZ/2\BZ$ corresponds to the two possible linearisations of the kernel. We will often write $\phi_{[n]}(\Phi)$ instead of $\phi_{[n]}((\Phi,1))$. 
Using the equivalence $\Psi$ we also denote the resulting homomorphism
\[
\Aut(\Db(S)) \times \BZ/2\BZ \hookrightarrow \Aut(\Db(S^{[n]}))
\]
obtained via conjugation by $\phi_{[n]}$. 

\begin{lem}
\label{lem:Ploog_map_on_coh}
	Let $\Phi\in \Aut(\Db(S))$ such that $\Phi^\h\in \mathrm{O}(\tH(S,\BZ))$ 
is the identity. Then $\phi_{[n]}(\Phi)$ acts trivially on the extended Mukai lattice $\tH(S^{[n]},\BQ)$.  
\end{lem}
\begin{proof}
	Let $\Phi=\FM_{\CE}$ and let us consider $\FM_{\CE^{\boxtimes n}} \in \Aut(\Db(S^n))$. Using \cite[Exc.\ 5.13]{HuybrechtsFM} and the Künneth formula, one sees that $\FM_{\CE^{\boxtimes n}}$ acts trivial on singular cohomology $\h^{\ast}(S^n,\BQ)$. 
	
	The line bundle $\CL_n \in \Pic(S^{[n]})$ corresponds to the equivariant 
object $(\CL^{\boxtimes n},1)$ in $\Dbn(S^n)$. By \cite[Prop.\ 2.3]{PloogSosnaCYHK} the equivalence $\phi_{[n]}(\Phi)$ sends $(\CL^{\boxtimes n},\pm 1)$ to the objects $(\Phi(\CL)^{\boxtimes n},\pm 1)$. Using the compatibility of Fourier--Mukai transforms with (equivariant) topological $K$-theory \cite[Sec.\ 6]{TaelmanDerHKLL}, one sees that $\phi_{[n]}(\Phi)$ induces an isomorphism of equivariant topological $K$-theory $\mathrm{K}^0_{\textup{top},\FS_n}(S^n)$ which fixes the  classes $[(\CL^{\boxtimes n},\pm 1)]$. 
	
	Moreover, the equivalence $\Psi$ induces an isomorphism $\mathrm{K}^0_{\textup{top},\FS_n}(S^n) \cong \mathrm{K}^0_{\textup{top}}(S^{[n]})$, see \cite[Ch.\ 10]{BKR} or \cite[Thm.\ 8.2]{TaelmanDerHKLL}. This implies that $\phi_{[n]}(\Phi)$ leaves the classes $v(\CL_n)$ and $v(\CL_n \otimes \CO_X(-\delta))$ in $\h^{\ast}(X,\BQ)$ invariant. Using the compatibility \eqref{diag:commut_T_epsilon} we see that the classes $\tilde{v}(\CL_n)$ and $\tilde{v}(\CL_n\otimes \CO(-\delta))$ are fixed by the action of $\phi_{[n]}(\Phi)$ on the extended Mukai lattice. To conclude the proof, simply observe that these classes generate $\tH(X,\BQ)$ as a $\BQ$-vector space, since $\Lambda_{LB}$ from Section~\ref{subsec:lattices} is a full rank lattice. 
\end{proof}
Let $\pi \colon \CS \to B$ be a smooth and proper family of K3 surfaces and consider a path $\gamma \colon [0,1]\to B$. This yields a parallel transport isometry $\h^{\ast}(\CS_{\gamma^{-1}(0)},\BZ)\cong \h^{\ast}(\CS_{\gamma^{-1}(1)},\BZ)$ of the fibres which we will also denote by $\gamma$. The family $\pi$ induces naturally a corresponding family $\pi^{[n]}\colon \CS^{[n]} \to B$ of relative Hilbert schemes over $B$. 
The path $\gamma$ in $B$ then gives for this deformation a corresponding parallel transport isometry $\gamma^{[n]}\colon\h^{\ast}(\CS_{\gamma^{-1}(0)}^{[n]},\BQ) \cong \h^{\ast}(\CS_{\gamma^{-1}(1)}^{[n]},\BQ)$.

Consider an element $g\in \DM(S)$ of the form $g=\gamma'\circ F\circ \gamma$. Here $\gamma$ respectively $\gamma'$ are as above parallel transport isometries obtained from deforming $S$ to $S'$ respectively $S''$ to $S$ and $F=f^{\h}$ for a Fourier--Mukai equivalence $f\colon \Db(S')\cong \Db(S'')$. We associate to $g$ the element $g^{[n]}\coloneqq\gamma'^{[n]}\circ \phi_{[n]}(f)^{\tH}\circ \gamma^{[n]}$. 
\begin{prop}
	\label{prop:map_DMonS_DMonSn}
	The association $g\mapsto g^{[n]}$ yields a well-defined group homomorphism
	\[
	d_n\colon \DM(S)\to \DM(S^{[n]}).
	\]
\end{prop}
\begin{proof}
	This follows as in the proof of Lemma~\ref{lem:Ploog_map_on_coh} together with the assertion of Lemma~\ref{lem:Ploog_map_on_coh}.
\end{proof}
\section{Auto-equivalences of Hilbert schemes}
\label{sec:autoeq_hilb}
Let $S$ be a projective K3 surface and $S^{[n]}$ be the $n$-th punctual Hilbert scheme. In this section we will calculate the action of certain auto-equivalences on $\tH(S^{[n]},\BQ)$. 
\subsection{Sign equivalence}
Denote by $F \in \Aut(\Dbn(S^n))$ the auto-equivalence given by tensoring 
with the sign-representation. It is the image of the generator of $\BZ/2\BZ$ under $\phi_{[n]}$. We will also denote by $F$ the auto-equivalence of $\Db(S^{[n]})$ induced via the equivalence $\Psi$. For a vector $v \in \tH^{1,1}(S^{[n]},\BQ)$ we denote by
\[
s_v \in \mathrm{O}(\tH(S^{[n]},\BQ)), \quad x\mapsto x-2 \frac{\tilde{b}(x,v)}{\tilde{b}(v,v)}v
\]
the Hodge isometry given by reflection along $v$. 
\begin{prop}
\label{lem:signrep_on_coh}
	The action of $F$ on $\tH(S^{[n]},\BQ)$ is given by $(-1)^{n+1}s_{\tde}$. 
\end{prop}
\begin{proof}
	For all topological line bundles $\CL \in K_{\textup{top}}^0(S)$ the involution $F$ exchanges the equivariant objects $(\CL^{\boxtimes n},1)$ and 
$(\CL^{\boxtimes n},-1)$ viewed as elements in equivariant topological $K$-theory $K_{\textup{top}, \FS_n}^0(S^n)$. Thus, by \eqref{eq:lb_different_linearization} the induced isometry $F^{\tH}$ on the extended Mukai lattice exchanges $\tilde{v}(\CL_{n})$ and $\tilde{v}(\CL_n\otimes \CO_{S^{[n]}}(-\delta))$. 
	
	If $n$ is odd, then we conclude from the above that for all $\lambda \in 
\h^2(S,\BZ)\subset \h^2(S^{[n]},\BZ)$ the action on the extended Mukai lattice $F^{\tH}$  satisfies
	\[
	\tal + \frac{\tde}{2} + \lambda + \left( 1+\frac{\tbb(\lambda,\lambda)}{2} \right)\beta \mapsto \tal - \frac{\tde}{2} + \lambda + \left( 1+\frac{\tbb(\lambda,\lambda)}{2} \right)\beta.
	\]
	This property completely characterizes $F^{\tH}$. 
	
	If $n$ is even, Lemma~\ref{lem:determinant_extended_rank} implies that the determinant of $F^{\tH}$ must be one, because $F$ preserves the rank of objects. The result then follows as for $n$ odd.
\end{proof}
\subsection{Spherical twist}
An object $\CE \in \Db(S)$ is called spherical if its $\Ext$-algebra satisfies $\Ext^{\ast}(\CE,\CE)\cong \h^{\ast}(S^2,\BC)$. The auto-equivalence $\ST_{\CE}$ given by the Fourier--Mukai functor $\FM_{\CG}$ with Fourier--Mukai kernel defined via the distinguished triangle
\[
\CE^{\vee}\boxtimes \CE\to \CO_{\Delta}\to \CG
\]
in $\Db(S\times S)$ is called the spherical twist \cite{SeidelThomas}. Its action on the Mukai lattice $\tH(S,\BZ)$ is given by the reflection $s_{v(\CE)}$. 

An important example is the spherical twist $\ST_{\CO_S}$ along the structure sheaf $\CO_S$. It induces on cohomology the reflection along the vector $\One + \pt$. The morphism $\phi_{[n]}$ yields an equivalence $P \in \Aut(\Db(S^{[n]}))$. 
\begin{prop}
	\label{prop:STonHilbn}
	The equivalence $P$ acts on $\tH(S^{[n]},\BQ)$ via the isometry $(-1)^{n+1}s_v$, where $v=\tal + \beta$.
\end{prop}
\begin{proof}
	We want to understand the images of line bundles under $P$. The spherical twist $\ST_{\CO_S}$ sends the structure sheaf $\CO_S$ to $\CO_S[-1]$. Applying \cite[Prop.\ 2.3]{PloogSosnaCYHK} we see that $P(\CO_{S^n},1)=(\CO_{S^n},-1)[-n]$ and $P(\CO_{S^n},-1)=(\CO_{S^n},1)[-n]$\footnote{The 
fact that the linearisations get exchanged follows from the Koszul sign convention for graded tensor products.}. Lemma~\ref{lem:determinant_extended_rank} shows that $\epsilon(P^{\tH})=1$ if $n$ is even. 
	
	We first consider the case when $n$ is odd. Assume there exists a smooth 
rational curve $C\subset S$ and let $\CL=\CO_S(C)$ be the corresponding 
line bundle with first Chern class $l\coloneqq\Rc_1(\CL)$. Then by Riemann--Roch $\CL$ has a unique section up to scaling and the higher cohomologies of $\CL$ vanish. The auto-equivalence $\ST_{\CO_S}$ sends the line bundle $\CL$ to $\CL|_C$. We infer that the equivariant object $(\CL^{\boxtimes n},-1)$ is being sent to $((\CL|_C)^{\boxtimes n},-1)$ under the auto-equivalence $P$. We want to transfer this identity to the Hilbert scheme via $\Psi$. 
From \eqref{eq:lb_different_linearization} we know that $\Psi((\CL^{\boxtimes n}),-1) \cong \CL_n \otimes \CO_{S^{[n]}}(-\delta)$. 

It is left to compute $\Psi((\CL|_C)^{\boxtimes n},-1)$\footnote{Thanks to the anonymous referee and Georg Oberdieck for spotting a mistake in an earlier version and Georg Oberdieck for discussions on computing images under $\Psi$.}. We claim $\Psi((\CL|_C)^{\boxtimes n},-1) \cong \iota_\ast \omega_{Z}$ for $ \iota \colon Z = C^{[n]} \cong \BP^n \subset S^{[n]}$. We will sketch the arguments, see also \cite[Sec.\ 3.2]{OberdieckLagrangianPlanes} for a thorough computation of this identity. 

The sheaf $\CO_C$ admits the resolution $\CL^\vee \to \CO_X$. Taking the $n$-th box product we obtain
\[
\left[ W^n(\CL^\vee) \to W^{n-1}(\CL^\vee) \to \dots \to W^1(\CL^\vee) \to W^0(\CL^\vee) \right] \cong (\CO_C^{\boxtimes n},1)
\]
where we used the notation as in \cite[Def.\ 3.4]{KrugRemarks}. Tensoring with the object corresponding to the sign-representation and envoking \cite[Lem.\ 3.3]{OberdieckLagrangianPlanes}
\begin{equation}
\label{eq:W_resolution_7_2}
\left[ W^0(\CL) \to W^{1}(\CL) \to \dots \to W^{n-1}(\CL) \to W^n(\CL) \right] \cong ((\CL|_C)^{\boxtimes n},-1).
\end{equation}
Applying $\Psi$ to \eqref{eq:W_resolution_7_2} and using \cite[Thm.\ 1.1]{KrugRemarks} we find
\begin{equation}
	\label{eq:koszul_resolution_7_2}
	\left[ \CO_{S^{[n]}} \to \CL^{[n]} \to \dots \to \bigwedge^{n-1}\CL^{[n]} \to \det(\CL^{[n]}) \right] \cong \Psi((\CL|_C)^{\boxtimes n},-1).
\end{equation}
In particular, the derived dual of $\Psi(\CO_X^{\boxtimes n},-1)$ is via \eqref{eq:koszul_resolution_7_2} identified with the Koszul resolution of a regular section of the bundle $\CL^{[n]}$ shifted by $[n]$. As the zero locus of this section is exactly $Z = C^{[n]}$ the claim follows from Grothendieck--Verdier duality. 
	
	Taking extended Mukai vectors and using Lemma~\ref{lem:prop_vectorizable} we see that $P^{\tH}$ sends the vector $\tilde{v}(\CL_n \otimes \CO_{S^n}(-\delta))=\tal -\frac{\tde}{2}+l$ to $w=\lambda +c\beta$ with $\lambda \in \h^2(S^{[n]},\BQ)$, because $\iota_\ast \omega_Z$ has rank $0$. We already know $P^{\tH}(\tilde{v}(\CO_{S^{[n]}}(-\delta)))=-\tilde{v}(\CO_{S^{[n]}})$ (we assume $n$ odd) and since $P^{\tH}$ is an isometry we conclude that
	\[
	c=\tbb(w,-\tilde{v}(\CO_{S^{[n]}}))=\tbb(\tilde{v}(\CL_n \otimes \CO_{S^{[n]}}(-\delta)),\tilde{v}(\CO_{S^{[n]}}(-\delta)))=\frac{-1-n}{2}.
	\] 
	Similarly we can use that $P^{\tH}(\tilde{v}(\CO_{S^{[n]}}))=-\tilde{v}(\CO_{S^{[n]}}(-\delta))$ to infer
	\[
	w=\lambda' - \frac{\delta}{2} + \frac{-1-n}{2}\beta=\lambda' -\frac{\tde}{2} -\beta
	\]
	with $\lambda'\in \h^2(S,\BQ) \subset \h^2(S^{[n]},\BQ)$. 
	
	Since $Z\cong \BP^n$, all curve classes on $Z$ are multiples of each other. A line in $Z$ is known to have homology class $l+ (n-1)\delta^{\vee}\in \h_2(X,\BZ)$ \cite[Ex.\ 4.11]{HassettTschinkelIntersectionnumbers}, where $\delta^{\vee}$ is the dual class to $\delta$ satisfying $\int_{S^{[n]}} \delta \delta^\vee = 1$. Denoting $s=l-\frac{\delta}{2}$ the cohomology class $s^{2n-1}\in \SH^{4n-2}(S^{[n]},\BQ)$ is Poincar\'e dual to a multiple of the homology class $l+(n-1)\delta^{\vee}$. Therefore, the degree $4n-2$ part in $\SH(S^{[n]},\BQ)$ of $\overline{v(\iota_\ast \omega_Z)}$ must be a multiple of $s^{2n-1}$. Since for $\mu \in \h^2(S^{[n]},\BQ)$ we have
	\[
	\psi(\mu^{2n-1})=\psi(e_{\mu}^{2n-1}(\One))=e_{\mu}^{2n-1}(\psi(\One))=b(\mu,\mu)^{n-1}\frac{(2n)!}{2^nn!}\mu\beta^{n-1}\in \Sym^n(\tH(X,\BQ))
	\]
	we conclude that $\lambda'=l$ and
	\[
	\tal + \frac{\tde}{2}+l \mapsto \frac{\tde}{2}+l -\beta.
	\]
	
	In general, for a class $l\in \h^2(S,\BZ)$ of square $-2$ there exists a 
deformation $S'$ of $S$ such that either $l$ or $-l$ is the class of a smooth rational curve $C'\subset S'$. Using Proposition~\ref{prop:map_DMonS_DMonSn} we can assume that the topological line bundle $\CL$ on $S$ with 
first Chern class $l$ is algebraic and that $\CL\cong \CO_S(C)$, where $C\subset S$ is a smooth rational curve. By the above, we therefore know the image of $\tilde{v}(\CL_n \otimes \CO_{S^{[n]}}(-\delta))$ under $P^{\tH}$. Since the vectors $\tilde{v}(\CL_n \otimes \CO_{S^{[n]}}(-\delta))$ for $\CL$ a topological line bundle on $S$ whose first Chern class has self-intersection $-2$ together with $\tilde{v}(\CO_{S^{[n]}})$ and $\tilde{v}(\CO_{S^{[n]}}(-\delta))$ generate the vector space $\tH(S^{[n]},\BQ)$, we have proven the assertion in the case that $n$ is odd.
	
	If $n$ is even, then the above shows that $P^{\tH}$ must be either $s_v$ 
or $-s_v$. Using that $\epsilon(P^{\tH})=\det(P^{\tH})=1$ yields the assertion.
\end{proof}
\begin{rmk}
\label{rem:Pn_in_Hilb}
	Here is one observation from the proof which might help to understand the extended Mukai lattice $\tH(S^{[n]},\BQ)$. 
	
	Given a smooth rational curve $C\subset S$ inside a K3 surface and the corresponding line bundle $\CL=\CO_S(C) \in \Pic(S)$ we have associated to it a line bundle $\CL_n\in \Pic(S^{[n]})$. Its Mukai vector $v(\CL_n)$ 
has self-pairing $n+1$ under the generalized Mukai pairing. We also associate to $\CL_n$ the class $\tilde{v}(\CL_n)\in \tH(S^{[n]},\BQ)$. This class has self-pairing $-(n+3)/2$. 
	
	The auto-equivalence $P$ induced from the spherical twist $\ST_{\CO_S}$ via Ploog's map $\phi_{[n]}$ sends the line bundle $\CL_n \otimes \CO_{S^{[n]}}(-\delta)$ to $\iota_\ast \omega_Z$. 
This is compatible with the pairings since the self-intersection of the projective space $\BP^n\cong C^{[n]}\subset S^{[n]}$ is $(-1)^n (n+1)$. The image of $\tilde{v}(\CL_n \otimes \CO_{S^{[n]}}(-\delta))$ under $P^{\tH}$ is 
	\[
	l-\frac{\delta}{2} +\frac{-1-n}{2}\beta=l-\frac{\tde}{2} - \beta. 
	\]
	Its self-intersection is equal to $-(n+3)/2$ which is exactly the value of $b(\ell,\ell)$, where $\ell$ is the class of a line in the projective space $C^{[n]}$ and we view a curve class as an element in $\h^2(S^{[n]},\BQ)$ via Poincar\'e duality. 
\end{rmk}
In the above prove we have calculated  $\Psi(\CO_C^{\boxtimes n},-1)$. One can also consider the image of $(\CO_C^{\boxtimes n},1)$ under $\Psi$, i.e.\ with the canonical linearization. One can show that this is $\CO_{Y_C}$, where $Y_C \subset S^{[n]}$ is the reducible subscheme which is the preimage of $C^{(n)} \subset S^{(n)}$ under the Hilbert--Chow morphism.  
\subsection{From K3 surfaces to Hilbert schemes}
We can now describe the homomorphism $d_n$ from Proposition~\ref{prop:map_DMonS_DMonSn}. 
Consider the natural inclusion $\tH(S,\BQ) \hookrightarrow \tH(S,\BQ) \oplus \BQ \delta = \tH(S^{[n]},\BQ)$ of quadratic spaces. For $g\in \RO(\tH(S,\BQ))$ we define $\iota(g) \in \RO(\tH(S^{[n]},\BQ))$ via $\iota(g)(\lambda) =g(\lambda)$ for $\lambda \in \tH(S,\BQ) \subset \tH(S^{[n]},\BQ)$ and $\iota(g)(\delta)=\delta$. This yields a group homomorphism
\[
\iota \colon \RO(\tH(S,\BQ)) \to \RO(\tH(S^{[n]},\BQ)).
\]
\begin{thm}
\label{thm:description_dn}
	The homomorphism $d_n\colon \DM(S)\to \DM(S^{[n]})$ is given by
	\[
	g\mapsto \det(g)^{n+1}  B_{-\delta/2} \circ \iota(g) \circ B_{\delta/2}.
	\]
\end{thm}
\begin{proof}
	The group $\DM(S)$ is equal to the group of orientation-preserving isometries $\RO^+(\tH(S,\BZ))$ of the full integral cohomology \cite{HLOYAutoK3,HMSK3Orientation}. This group is generated by the reflection along the $-2$-vector $\One + \pt$ and the isometries $B_\lambda$ for $\lambda\in \h^2(S,\BZ)$ \cite[Prop.\ 3.4]{GHS_Pi1}. Note that the assignment of the statement of the theorem does define a group homomorphism $\RO(\tH(S,\BQ))\to \RO(\tH(S^{[n]},\BQ))$. Hence, one can check on generators of $\DM(S)$ that this morphism agrees with $d_n$. This is a straightforward calculation.
\end{proof}
\section{Invariant Lattice}
\label{sec:computation_upper_bound_DMON_K3n}
Let $X$ be again a $\Kdrein$-type hyper-Kähler manifold with $n>1$. Any $\Gamma \cong \BZ^{25}$ with an inclusion $\Gamma \hookrightarrow \tH(X,\BQ)$ inherits a quadratic form which takes values in the rational numbers. We will denote by $\RO(\Gamma)\subset \RO(\tH(X,\BQ))$ the group of 
all isometries $\gamma$ satisfying $\gamma(\Gamma)= \Gamma$. 

The main goal of this section is to prove the following result. 

\begin{thm}
\label{thm:derived_mon_grp_K3n}
	Let $X$ be a $\Kdrein$-type hyper-Kähler manifold. There are inclusions
	\[
	\hat{\RO}^+(\Lambda)\subset \DM(X) \subset \RO(\Lambda).
	\]
	In particular, the $\Kdrein$ lattice $\Lambda$ is fixed by all derived equivalences. 
\end{thm}
The group $\hat{\RO}^+(\Lambda)$ is the group of all isometries with spinor norm 1 and which act via $\pm \id$ on the discriminant group. For $n=2$ this result was also obtained by Taelman \cite[Thm.\ 9.8]{TaelmanDerHKLL}.

\subsection{Realizing orthogonal transformations as derived equivalences}
The first inclusion follows easily from the results of the last sections.

\begin{prop}
\label{prop:first_incl_DM}
	There is an inclusion
	\[
	\hat{\RO}^+(\Lambda)\subset \DM(X).
	\]
\end{prop}
\begin{proof}
	The shift $[1]$ acts on the extended Mukai lattice by $-\id$ and therefore acts non-trivially on the discriminant lattice and has determinant $-1$. Proposition~\ref{prop:STonHilbn} endows us with an isometry whose action on the discriminant lattice is trivial if and only if its determinant is non-trivial and vice versa. Hence, it suffices to show that $\widetilde{\mathrm{SO}}^+(\Lambda)$, i.e.\ the group of all isometries with spinor norm and determinant 1 acting trivially on the discriminant, is contained in $\DM(X)$. For this we will use the notion of Eichler transvections, for details and notations see \cite[Sec.\ 3]{GHS_Pi1}.
	
	Let us orthogonally decompose \[
	\Lambda = U \oplus \Lambda'
	\]
	where the hyperbolic plane $U$ is spanned by $\tal$ and $-\beta$. The group $\widetilde{\mathrm{SO}}^+(\Lambda)$ equals the group $E_U(\Lambda')$ 
of unimodular transvections \cite[Prop.\ 3.4]{GHS_Pi1}. For $\lambda \in \Lambda'$ the Eichler transvection $t(-\beta,\lambda)$ equals $B_{\lambda}$ (note that for $\tde \in \Lambda'$ the transvection $t(-\beta,\tde)$ also equals $B_{\delta}$). Using tensoring with line bundles we see that all these isometries are contained in $\DM(X)$. Furthermore, we infer from 
Proposition~\ref{prop:STonHilbn} that the reflection $s_v$ along the vector $v=\tal + \beta$ lies in $\DM(X)$. This involution exchanges $\tal$ and $-\beta$ and acts trivially on $\Lambda'$. Using \cite[Eq.\ (6)]{GHS_Pi1} we deduce that the transvections $t(\tal,\lambda)$ for $\lambda \in \Lambda'$ are contained in $\DM(X)$. By \cite[Prop.\ 3.4]{GHS_Pi1} these isometries generate $\widetilde{\mathrm{SO}}^+(\Lambda)$ yielding the assertion.
\end{proof}
\subsection{Finding derived invariant lattices}
The proof of the other inclusion in Theorem~\ref{thm:derived_mon_grp_K3n} 
will occupy the remainder of this section. 
\begin{lem}
\label{lem:exist_lattice}
	There exists a lattice $\Gamma  \hookrightarrow \tH(X,\BQ)$ of rank $25$ 
such that $\DM(X) \subset \mathrm{O}(\Gamma)$.
\end{lem}
\begin{proof}
	The group $\DM(X)$ has a natural and faithful action on $\SH(X,\BQ)$ via 
the embedding $\DM(X) \subset \RO(\tH(X,\BQ))$. Moreover, $\DM(X)$ preserves the integral lattice $v(K_{\textup{top}}(X)))\cap \SH(X,\BQ) \subset \SH(X,\BQ)$ in this representation. Therefore it is contained in an arithmetic subgroup of $\RO(\tH(X,\BQ))$. 
\end{proof}

We want to classify lattices $\Gamma$ with the property $\DM(X)\subset \RO(\Gamma)$. 
We know by Proposition~\ref{prop:first_incl_DM} that for any such lattice 
$\Gamma$ there is an inclusion $\hat{\RO}^+(\Lambda)\subset \RO(\Gamma)$. 
This yields strong restrictions. 
\begin{lem}
\label{lem:form_of_lattice}
	Let $\tilde{\Gamma}$ be a lattice preserved by $\DM(X)$ as in Lemma~\ref{lem:exist_lattice}. Up to replacing $\tilde{\Gamma}$ by $a\tilde{\Gamma} 
\subset \tH(X,\BQ)$ for $a\in \BQ$ the lattice $\tilde{\Gamma}$ is equal (as subsets) to $k \Lambda_S \oplus \BZ \tde \subset \tH(X,\BQ)$ for some $k\in \BZ$ satisfying $k|(2n-2)$. 
\end{lem}
\begin{proof}
	Let us replace $\tilde{\Gamma}$ with $a\tilde{\Gamma}$ for $a\in \BQ_{>0}$ such that $\tilde{\Gamma} \subset \Lambda$ and $a$ is the smallest positive rational number with that property.
	
	Consider $v\in \tilde{\Gamma}$ and write $v=x+b\tde$ with $x\in \Lambda_S$ and $b\in \BZ$. If $x \neq 0$, then its divisibility agrees with the largest integer $t\in \BZ_{>0}$ such that $x \in t\Lambda_S$, since $\Lambda_S$ is unimodular. Consider now all $v\in \tilde{\Gamma}$ such that in the above decomposition $x\neq 0$ and let $k$ be the minimum of all integers $t$ as above. Then $k\Lambda_S \subset \tilde{\Gamma}$.
	
	Indeed, take an element $v\in \tilde{\Gamma}$ such that $v= kx + c\tde$ for some $c\in \BZ$ and $x\in \Lambda_S$ is primitive. One immediately sees that $\RO^+(\Lambda_S)$ can be embedded into $\hat{\RO}^+(\Lambda)$ as the group of all isometries fixing $\tde$. Using \cite[Prop.\ 3.3]{GHS_Pi1} we see that for every primitive $y\in \Lambda_S$ with $\tbb(y,y)=\tbb(x,x)$ the element $ky+c\tde$ is contained in $\tilde{\Gamma}$. This yields $k\Lambda_S \subset \tilde{\Gamma}$. 
	
	Consider $(k\Lambda_S)^\perp \subset \tilde{\Gamma}$ and take the positive integer $s\in \BZ$ such that $(k\Lambda_S)^\perp = s\BZ\tde \subset \tilde{\Gamma}$. We claim $\tilde{\Gamma} = k\Lambda_S \oplus s\BZ\tde$. For this take an arbitrary $v\in \tilde{\Gamma}$ and write $v=dx+e\tde$ for $d,e\in \BZ$. The definition of the integer $k$ implies that $k$ 
divides $d$ and by the above we therefore have that $dx\in k\Lambda_S \subset \tilde{\Gamma}$. Hence $v-dx=e\tde$ is an element of $\tilde{\Gamma}$ orthogonal to $k\Lambda_S$. By definition of the integer $s$ we have that $s$ divides $e$ and so $v\in k\Lambda_S\oplus s\BZ \tde$. 
	
	The minimality assumption of $a$ yields that the integers $k$ and $s$ do 
not have a common divisor. On the other hand, we know that $k\tal\in \tilde{\Gamma}$ and $B_{\delta}(k\tal) = k\tal +k\tde +k(1-n)\beta$. This implies that $k\tde\in \tilde{\Gamma}$ and therefore $s=1$. 
	Finally, one sees that $B_{\delta}(\tde)=\tde + (2-2n)\beta$ which finishes the proof.
\end{proof}

\begin{rmk}
Ideally, one would like to conclude in the above situation directly that $k=1$ and therefore (up to scaling) $\tilde{\Gamma}$ must equal $\Lambda$. However, this is in general not true. 

For example, let us consider the case of $\mathrm{K3}^{[10]}$-type hyper-Kähler manifolds. The lemma below implies that for the lattice $\tilde{\Gamma}=3\Lambda_S \oplus \BZ\tde \subset \tH(X,\BQ)$ there is an inclusion $\RO(\Lambda)\subset \RO(\tilde{\Gamma})$. Moreover, the isometry 
$B_{\delta/3}$ lies in $\RO(\tilde{\Gamma})$ but not in $\RO(\Lambda)$. Therefore additional (geometric) input is necessary for the proof of Theorem~\ref{thm:derived_mon_grp_K3n}. 
\end{rmk}

We make some further reductions. 
\begin{lem}
\label{lem:lattice_reduction_lsquared}
	Let $l\in \BZ_{>0}$ be the largest integer such that $l^2|(n-1)$. For every lattice $\tilde{\Gamma}$ as in Lemma~\ref{lem:form_of_lattice} there is an inclusion
	\[
	\RO(\tilde{\Gamma}) \subset \RO(\Gamma)
	\]
	with $\Gamma\coloneqq l\Lambda_S\oplus \BZ \tde\subset \tH(X,\BQ)$.
\end{lem}
\begin{proof}
	Write $\tilde{\Gamma} = k\Lambda_S \oplus \BZ \tde$ with $k|(2n-2)$. Let $t$ be the greatest common divisor of $l$ and $k$ and denote by $\Gamma_t$ the lattice $t\Lambda_S \oplus \BZ \tde \subset \tH(X,\BQ)$. The proof consists of showing the following two inclusions
	\[
	\RO(\tilde{\Gamma}) \subset \RO(\Gamma_t)\subset \RO(\Gamma).
	\]
	
	Let us prove the first inclusion. Take an isometry $\gamma \in \RO(\tilde{\Gamma})$ and write $k=k't$. Since $\gamma(\tde) \in \tilde{\Gamma} \subset \Gamma_t$ as subsets of the extended Mukai lattice it suffices to show that for every $\lambda \in \Lambda_S$ we have $\gamma(t\lambda)\in \Gamma_t$. By definition we have $\gamma(k\lambda)\in \tilde{\Gamma}$. Therefore we can write $\gamma(k\lambda)= a k\mu +b\tde$ for some integers $a$ and $b$ and $\mu \in \Lambda_S$. Dividing this equation by $k'$ we obtain
	\[
	\gamma(t\lambda)= at\mu + \frac{b}{k'}\tde.
	\]
	As the self-pairing of $t\lambda$ is an even integer, the same must hold 
true for $\gamma(t\lambda)$. In particular, we find that
	\[
	2b^2 \frac{1-n}{k'^2}\in 2\BZ. 
	\]
	The defining property of $l$ together with the fact that $l$ and $k'$ are coprime implies that $k'$ must divide $b$. This gives the first inclusion.
	
	For the second inclusion consider an isometry $\gamma \in \RO(\Gamma_t)$ 
and observe that for every $\lambda \in \Lambda_S$ we have $t\gamma(\lambda)\in \Gamma_t=t\Lambda_S \oplus \BZ \tde$. This yields 
	\[
	l\gamma(\lambda)=(l/t)t\gamma(\lambda)\in l\Lambda_S \oplus \BZ\tde = 
\Gamma.
	\]
	It is left to show that $\gamma(\tde)\in \Gamma$. This follows immediately from the fact that $\tde$ as an element in the lattice $\Gamma_t$ has divisibility $2n-2$.
\end{proof}
We therefore have an upper bound for the lattice from Lemma~\ref{lem:exist_lattice}, i.e.\ for $\Gamma=l\Lambda_S \oplus \BZ \tde$ as above we have $\DM(X)\subset \RO(\Gamma)$. 
In particular, if $n-1$ is square-free, then we have already obtained $\DM(X)\subset \RO(\Lambda)$. 
\subsection{Conclusion of proof}
\begin{proof}[Proof of Theorem~\ref{thm:derived_mon_grp_K3n}]
	From Lemma~\ref{lem:lattice_reduction_lsquared} we know that for the lattice $\Gamma=l\Lambda_S \oplus \BZ\tde$ with $l$ maximal such that $l^2|(n-1)$ there is an inclusion $\DM(X)\subset \RO(\Gamma)$. 
	
	Suppose there exists an isometry $\gamma \in \DM(X)$ which does not lie in $\RO(\Lambda)$. Consider the composition \[
	\varphi\colon \Lambda_S \xrightarrow{\gamma} \tH(X,\BQ)\xrightarrow{p} \BQ \tde
	\]
	where $p$ is the orthogonal projection and denote $K=\Ker(\varphi)$. Let $v$ be a generator of $K^{\perp}\subset \Lambda_S$ and let us write $\frac{k}{l}\tde$ for its image under $\gamma$. By assumption $\frac{k}{l}$ is not an integer. 
	Note that there are two hyperbolic planes $U_1 \oplus U_2$ contained in $K$. 
	
	Indeed, since $\RO^+(\Lambda_S)$ acts transitively on primitive elements 
with the same square \cite[Prop.\ 3.3]{GHS_Pi1}, one can send $v$ into a hyperbolic plane $U\subset U^4 \oplus E_8(-1)^2 \cong \Lambda_S$. Since $K=v^{\perp}$ we know there are two (in fact at least three) hyperbolic planes contained in $K$.
	
	Changing $v$ to $w$ by adding an element of $U_1$ we can assume that $\tbb(w,w)=-2$ and the image of $w$ generates the image of $\varphi$. Moreover, we know that there is a primitive isotropic element $z\in U_2$ which is orthogonal to $w$ and which under $\gamma$ is mapped to a primitive element $u \in \Lambda_S$.
	
	Let us write 
	\[
	\gamma(w)=x+\frac{k}{l}\tde
	\]
	with $x\in \Lambda_S$ and
	\begin{equation}
	\label{eq:aabse}
	\gamma(\tde)=\frac{2n-2}{l}y + s \tde
	\end{equation}
	for $y\in \Lambda_S$ and some $s\in \BZ$, because $\tde\in \Gamma$ has divisibility $2n-2$. 
	Recall $u=\gamma(z)$ and note that there exists $a\in \BZ$ such that
	\[
	x'=x+\frac{n-1}{l}y+au \in \Lambda_S
	\]
	is primitive since $u\in \Lambda_S$ is itself primitive. We define the element $w'=w + az\in \Lambda_S$ which is still primitive, has self-pairing $-2$ and its image generates the image of $\varphi$.
	
	The group $\RO^+(\Lambda_S)$ can be included into $\hat{\RO}^+(\Lambda) \subset \DM(X)$ by letting isometries act trivially on $\tde$. Hence, there exists an element $\gamma'\in \hat{\RO}^+(\Lambda) $ which maps $\tal + \beta$ to $w'\in \Lambda_S$ and fixes $\tde$, since $\RO^+(\Lambda_S)$ acts transitively on the set of primitive vectors with prescribed self-pairing. 
	Furthermore, there exists an isometry $\gamma''\in \hat{\RO}^+(\Lambda)$ 
such that the primitive element $x'$ is mapped to $\tal +b\beta$ for some 
$b\in \BZ$ and $\tde$ to itself. Precomposing $\gamma$ with $\gamma'$ and 
postcomposing with $\gamma''$ we therefore obtain an isometry in $\DM(X)$ 
which satisfies
	\[
	\tilde{v}(\CO_X)=\tal + \frac{\tde}{2} + \beta \mapsto h \coloneqq \tal + \left( \frac{s}{2}+\frac{k}{l} \right) \tde + b\beta=\alpha + \left(\frac{s-1}{2}+\frac{k}{l}\right)\delta + c\beta
	\]
	for some $c\in \BQ$. 
	
	The extended Mukai vector of $\CO_X$ satisfies
	\[
	T\left(\frac{\tilde{v}(\CO_X)^n}{n!}\right) = \overline{v(\CO_X)} \in \SH(X,\BQ) \subset \h^{\ast}(X,\BQ)
	\]
	which is in particular an element in the image of the Mukai vector morphism
	\[
	\overline{v}=\overline{\ch(\_)\tdd}\colon K^0_{\textup{top}}(X)\to \SH(X,\BQ) \subset \h^{\ast}(X,\BQ)
	\]
	projected to $\SH(X,\BQ)$. Since parallel transport operators as well as 
derived equivalences preserve the image of topological $K$-theory under the Mukai vector morphism in cohomology, the same must hold true for $\gamma''\circ \gamma \circ \gamma'$, so in particular
	\[
	T\left(\frac{h^n}{n!}\right) \in \overline{v(K^0_{\textup{top}}(X))}.
	\]
	Let us write
	\[
	T\left(\frac{h^n}{n!}\right) = \One + \left(\frac{s-1}{2}+\frac{k}{l}\right)\delta + \mu
	\]
	with $\mu \in \SH^{>2}(X,\BQ)$. Applying the quadratic form to the equality \eqref{eq:aabse} we see that $s$ must be odd and therefore $(s-1)/2$ is an integer. Note that the degree 2 component of the image of elements of topological $K$-theory under the Mukai vector morphism always lies inside $\h^2(X,\BZ)$. This yields a contradiction and finishes the proof. 
\end{proof}

\begin{cor}
\label{conj:der_inv_geom_lattice}
	There is an inclusion
	\[
	\DM(X)\subset \RO(\Lambda_g).
	\]
\end{cor}
\begin{proof}
	Take $\gamma\in \DM(X)$ and recall that
	\[
	\Lambda \subset \Lambda_g = \Lambda_S \oplus \BZ \frac{\tde}{2}.
	\]
	The inclusion $\DM(X)\subset \RO(\Lambda)$ yields that every element of $\Lambda_S$ is mapped under $\gamma$ again to $\Lambda \subset \Lambda_g$. Moreover,
	\[
	\gamma(\tde)=(2n-2)x+s\tde
	\] 
	for some $s\in \BZ$ and $x\in \Lambda_S$, because $\tde\in \Lambda$ has divisibility $2n-2$. This implies $\gamma(\tde/2) \in\Lambda_g$.
\end{proof}
\section{Derived equivalences of $\Kdrein$-type hyper-Kähler manifolds}
\label{sec:conclusion_K3nHK_Db}
We will draw some consequences from the results of the previous sections. 

\subsection{General results}
Let $X$ be a projective hyper-Kähler manifold of $\Kdrein$-type. We denote by $\Lambda_X$ the Hodge structure obtained from the inclusion $\Lambda\subset \tH(X,\BQ)$ and by $\Aut(\Lambda_X)$ the group of all Hodge isometries of $\Lambda_X$. Recall the representation
\[
\rho^{\tH} \colon \Aut(\Db(X)) \to \mathrm{O}(\tH(X,\BQ))
\]
from Section~\ref{subsec:prelim_der_equi_tael}. 
\begin{cor}
\label{cor:Representation_autoeq_K3n_Lattice}
	The representation $\rho^{\tH}$ of the group of auto-equivalences $\Aut(\Db(X))$ factors via a representation
	\[
	\rho^{\tH}\colon \Aut(\Db(X))\to \Aut(\Lambda_X) \subset \mathrm{O}(\tH(X,\BQ)).
	\]
\end{cor}
One can also formulate the following more general version of the above statement.
\begin{thm}
\label{prop:rel_version_hodge_isom_lattice}
	Let $X$ and $Y$ be projective $\Kdrein$-type hyper-Kähler manifolds 
and $\Phi\colon \Db(X)\cong \Db(Y)$ a derived equivalence. Then $\Phi^{\tH}$ restricts to a Hodge isometry
	\[
	\Phi^{\tH} \colon \Lambda_X\cong \Lambda_Y.
	\]
\end{thm}
\begin{proof}
	Since $X$ and $Y$ are deformation-equivalent, there exists a parallel transport isometry $\gamma \colon \tH(Y,\BQ)\cong \tH(X,\BQ)$. The composition $\Phi^{\tH}\circ \gamma$ lies in $\DM(Y)$ and, therefore, satisfies $\Phi^{\tH}\circ \gamma(\Lambda) = \Lambda$ by Theorem~\ref{thm:derived_mon_grp_K3n}. Now $\gamma(\Lambda)\subset \tH(X,\BQ)$ is a lattice invariant under $\DM(X)$.
	
	Indeed, let $\gamma_2 \circ F \circ \gamma_1$ be one of the generators of $\DM(X)$ as in Definition~\ref{defn:derived_mon_grp}. Then $\gamma^{-1} \circ \gamma_2 \circ F \circ \gamma_1 \circ \gamma \in \DM(Y)$, so Theorem~\ref{thm:derived_mon_grp_K3n} gives
	\[
	\gamma^{-1} \circ \gamma_2 \circ F \circ \gamma_1 \circ \gamma (\Lambda) = \Lambda
	\]
	which yields
	\[
	\gamma_2 \circ F \circ \gamma_1 (\gamma (\Lambda)) = \gamma(\Lambda).
	\]
	
	Using Lemma~\ref{lem:form_of_lattice} and that $\gamma$ is an isometry we see that the subset $\gamma(\Lambda)\subset\tH(X,\BQ)$ is equal to $\Lambda \subset \tH(X,\BQ)$. Combining everything yields the assertion.
\end{proof}
We can use Lemma~\ref{lem:transcendental_hodge_structure} to obtain the following form of Theorem~\ref{prop:rel_version_hodge_isom_lattice}.
\begin{cor}
\label{cor:Der_equi_yields_transcendental_isometry}
	Let $X,Y$ and $\Phi$ be as above. Then $\Phi^{\tH}$ restricts to a Hodge 
isometry
	\[
	\Phi^{\tH} \colon \h^2(X,\BZ)_{\textup{tr}} \cong \h^2(Y,\BZ)_{\textup{tr}}
	\]
	between the transcendental lattices of $X$ and $Y$.\qed
\end{cor}
An immediate consequence is the following.
\begin{thm}
\label{prop:FM_Partners_finite}
	For a fixed projective $\Kdrein$-type hyper-Kähler manifold $X$ the 
number of projective $\Kdrein$-type manifolds $Y$ up to isomorphism with $\Db(X)\cong \Db(Y)$ is finite.
\end{thm}
\begin{proof}
	The proof is similar in flavour to \cite[Prop.\ 5.3]{BridgelandMaciociaSurfaces}. 
	
	Corollary~\ref{cor:Der_equi_yields_transcendental_isometry} implies that 
for any $Y$ as in the assertion its transcendental lattice $\h^2(Y,\BZ)_{\textup{tr}}$ is Hodge isometric to $\h^2(X,\BZ)_{\textup{tr}}$. As abstract lattices the number of embeddings
	\[
	\h^2(Y,\BZ)_{\textup{tr}}\hookrightarrow\h^2(Y,\BZ)
	\]
	is finite up to isometries of $\h^2(Y,\BZ)$, see \cite[Satz 30.2]{KneserQuadratischeFormen}. Therefore, the set of lattices appearing as $\textup{NS}(Y)$ for any such $Y$ is as well finite. 
	
	As in \cite[Prop.\ 5.3]{BridgelandMaciociaSurfaces} we conclude that there are only finitely many Hodge structures on the lattice $\h^2(Y,\BZ)$ being realized by $\Kdrein$-type hyper-Kähler manifolds $Y$ derived equivalent to our fixed $X$. Since the monodromy group $\mathrm{Mon}^2(Y)$ 
is a finite index subgroup of $\RO(\h^2(Y,\BZ))$ \cite[Cor.\ 1.8]{MarkmanMonodromyModuli} the Global Torelli Theorem \cite{VerbitskyTorelli} shows 
that up to birational equivalence there are only finitely many hyper-Kähler manifolds realizing a given Hodge structure on $\h^2(Y,\BZ)$. The 
assertion now follows from \cite[Cor.\ 1.5]{MarkmanYoshiokaKMConjecture}. 

\end{proof}
We also have the following structural result.
\begin{cor}
	\label{cor:prop_kernel_der_equi}
	Let $X,Y$ and $\Phi$ be as in Theorem~\ref{prop:rel_version_hodge_isom_lattice} and let $\CE$ be the Fourier--Mukai kernel of $\Phi$. Then the rank of $\CE$ is of the form $n!a^n$ for $a\in \BZ$ and the smallest non-zero cohomological degree of the Mukai vector of 
the image of $k(x)$ under $\Phi$ for all $x\in X$ is $0,2n$ or $4n$. 
In the second case $Y$ admits a rational Lagrangian fibration.
\end{cor}
\begin{proof}
	The first statements follow from Lemma~\ref{lem:properties_kx_vectorizable} and Theorem~\ref{prop:rel_version_hodge_isom_lattice}. 
	
	For the last assertion, we know that $\beta\in \Lambda_X$ is mapped to $\lambda + c\beta\in \Lambda_Y$ for $c\in \BZ$ and $\lambda \in \h^{1,1}(Y,\BZ)$ satisfying $b(\lambda,\lambda)=0$. Let $\CC_Y \subset \h^2(Y,\BR)$ the the positive cone of $Y$. Then $\lambda^\perp \cap \CC_Y \neq 0$. By \cite{MarkmanSurvey} there exist an isometry mapping $\lambda$ into the closure of the birational Kähler cone. 
	The result now follows from \cite[Cor.\ 1.1]{MatsushitaOn}. 
\end{proof}
This yields strong restrictions on Fourier--Mukai kernels of derived equivalences between hyper-Kähler manifolds of $\Kdrein$-type. Note that 
all three cases $0,2n$ and $4n$ occur, see Proposition~\ref{prop:STonHilbn} and Section~\ref{subsec:Example_POINCARE_ADM}. Furthermore,  Lemma~\ref{lem:properties_kx_vectorizable} implies that if $\CE$ if of rank zero, then for all $x\in X$ all Chern classes of $\CE_x$ are isotropic as in Lemma~\ref{lem:properties_kx_vectorizable}. 
\subsection{Moduli spaces} We demonstrate consequences for smooth moduli spaces of stable objects, see \cite[Ch.\ 10]{HuybrechtsK3} and \cite{BayMacMMP} for the necessary background and notation. 
\begin{cor}
\label{cor:der_equi_moduli_spaces}
	Let $M^S_{\sigma}(v)$ be a smooth moduli space of stable objects on a projective K3 surface $S$ and $X$ a projective $\Kdrein$-type hyper-Kähler manifold such that $\Db(X)\cong \Db(M^S_{\sigma}(v))$. Then $X$ is itself a moduli space of stable objects on $S$. 
\end{cor}
\begin{proof}
	For a K3 surface $S$ and a primitive Mukai vector $v$ with generic stability condition $\sigma \in \Stab^{\dagger}(S)$ one has a Hodge isometry
	\[
	\h^2(M^S_{\sigma}(v),\BZ)\cong v^{\perp} \subset \tH(S,\BZ),
	\]
	see \cite{YoshiokaAbelian}, \cite[Thm.\ 6.10]{BayMacPrBiGeo} and \cite[Thm.\ 1.1 (2)]{BottiniStableSheavesWallCrossing}. Since $v$ is in the algebraic part of the Mukai lattice of the K3 surface the restriction of the above Hodge isometry yields
	\[
	\h^2(M^S_{\sigma}(v),\BZ)_{\textup{tr}}\cong \h^2(S,\BZ)_{\textup{tr}}.
	\]
	
	Corollary~\ref{cor:Der_equi_yields_transcendental_isometry} together with \cite[Prop.\ 4]{AddingtonTwoConj} and \cite[Thm.\ 1.2(c)]{BayMacMMP} imply 
that $X$ is a moduli space $M^{S'}_{\sigma'}(v')$ of stable objects on a K3 surface $S'$ such that
	\[
	\h^2(S,\BZ)_{\textup{tr}}\cong \h^2(S',\BZ)_{\textup{tr}}.
	\]
	From \cite[Thm.\ 3.3]{OrlovEquiFM} we infer that $S$ and $S'$ are derived equivalent. Choosing one such equivalence $\Phi \colon \Db(S')\cong \Db(S)$ yields an isomorphism
	\[
	\Phi \colon M^{S'}_{\sigma'}(v')\cong M^S_{\Phi.\sigma'}(\Phi^{\h}(v'))
	\]
	where the latter variety is a moduli space of stable objects on $S$.
\end{proof}
The proof of the corollary also shows the following.
\begin{cor}
\label{cor:der_equi_K3_Hilbn}
	For two smooth moduli spaces $M^{S}_{\sigma}(v)$ and $M^{S'}_{\sigma'}(v')$ of stable objects on projective K3 surfaces $S$ and $S'$ with $\Db(M^S_{\sigma}(v))\cong \Db(M^{S'}_{\sigma'}(v'))$ we have $\Db(S)\cong \Db(S')$. Furthermore, $S$ and $S'$ are derived equivalent if and only if their Hilbert schemes $S^{[n]}$ and $S'^{[n]}$ are derived equivalent. 
\end{cor}
\begin{proof}
	The first part follows from the above and that derived equivalent K3 surfaces have derived equivalent Hilbert schemes was proven in \cite[Prop.\ 8]{PloogEquivFinite}. 
\end{proof}
\subsection{Hilbert schemes} 
We specialize to elliptic K3 surfaces $S$ with a section and their Hilbert schemes. Recall that an elliptic K3 surface $S$ has a section if and only if $U \subset \textup{NS}(S)$ \cite[Rem.\ 11.1.4]{HuybrechtsK3}. 
Theorem~\ref{thm:description_dn} allows us to determine in this situation 
the image of the representation $\rho^{\tH}$ up to finite index.
\begin{thm}
\label{prop:image_representation_Hilb_U}
	For the Hilbert scheme $S^{[n]}$ of a K3 surface with $U\subset \textup{NS}(S)$ the image $\textup{Im}(\rho^{\tH})$ of the representation $\rho^{\tH}$ satisfies
	\[
	\hat{\Aut}^+(\Lambda_{S^{[n]}})\subset \textup{Im}(\rho^{\tH}) \subset \Aut(\Lambda_{S^{[n]}}).
	\]
\end{thm}
The group $\hat{\Aut}^+(\Lambda_{S^{[n]}})$ is the group of all Hodge isometries with real spinor norm one which act via $\pm \id$ on the discriminant group. 
\begin{proof}
	Let $\gamma\in \Aut^+(\Lambda_{S^{[n]}})$ be a Hodge isometry with real spinor norm one which acts 
trivially on the discriminant group. We want to show $\gamma \in \textup(Im)(\rho^{\tH})$.

For line bundles $\CL\in \Pic(S^{[n]})$ the auto-equivalence $\SM_\CL$ given by tensoring with $\CL$ as well as the equivalence $\phi_{[n]}(\ST_{\CO_S})$ from Proposition~\ref{prop:STonHilbn} are contained in $\Aut(\Db(S^{[n]}))$. 
The assumption $U \subset \textup{NS}(S)$ implies that $\Lambda_{S^{[n]},\textup{alg}}$ contains two copies of the hyperbolic plane $U$. The elements $\tde$ and $\gamma(\tde)$ are both contained in $\Lambda_{S^{[n]},\textup{alg}}$ and have the same self-pairing as well as divisibility. As explained in the proof of Proposition~\ref{prop:first_incl_DM} using \cite[Sec.\ 3]{GHS_Pi1} we conclude that there exists a derived equivalence $\Phi\in\Aut(\Db(S^{[n]}))$ whose induced action $\Phi^{\tH}$ is trivial on the discriminant group, has real spinor norm one and sends $\gamma(\tde)$ to $\tde$, i.e.\
\[
\Phi^{\tH} \circ \gamma (\tde) = \tde. 
\]

In particular, the isometry $\Phi^{\tH} \circ \gamma$ restricts to a Hodge isometry of 
\[
\tde^\perp = B_{-\delta/2} (\h^2(S,\BZ)) \subset \Lambda_{S^{[n]}}
\]
with real spinor norm one. Using \cite[Cor.\ 3]{HMSK3Orientation} there is an auto-equivalence $\eta \in \Aut(\Db(S))$ such that
\[
B_{-\delta/2} \circ \Phi^{\tH} \circ \gamma \circ B_{\delta/2} 
\]
restricted to a Hodge isometry of $\tH(S,\BZ)$
agrees with $\eta^{\tH}$. 
Theorem~\ref{thm:description_dn} implies that $\Phi^{\tH}\circ \gamma$ or $-(\Phi^{\tH}\circ \gamma)$ lies in $\textup{Im}(\rho^{\tH})$. As the shift functor $[1]$ acts as $-\id$, we conclude that $\Phi^{\tH}\circ \gamma \in \textup{Im}(\rho^{\tH})$ and, therefore, $\gamma \in \textup{Im}(\rho^{\tH})$.

Hence, we have proven that all Hodge isometries with real spinor norm one which act trivially on the discriminant lattice are contained in $\textup{Im}(\rho^{\tH})$. The assertion now follows from Proposition~\ref{lem:signrep_on_coh} which yields an isometry acting as $-\id$ on the discriminant group. 
\end{proof}
\begin{prop}
	Let $X$ be a projective $\Kdrein$-type hyper-Kähler manifold such that $\Db(X)\cong \Db(S^{[n]})$ for a K3 surface $S$ with $U\subset \textup{NS}(S)$. Then $X$ and $S^{[n]}$ are birational.
\end{prop}
\begin{proof}
	The derived equivalence yields a Hodge isometry
	\[
	\varphi\colon \Lambda_X \cong \Lambda_{S^{[n]}}
	\]
	which by \cite[Prop.\ 3.3]{GHS_Pi1} and $U\subset \textup{NS}(S)$ we can 
postcompose by a Hodge isometry to assume $\varphi(\beta)=\beta$.  Therefore the preimage of $\tde\in \Lambda_{S^{[n]},\textup{alg}}$ under the isometry $\varphi$ must be of the form 
	\[\gamma +c(2n-2)\beta\in \Lambda_{X,\textup{alg}}\] 
	for some $c\in \BZ$ and $\gamma\in \h^{1,1}(X,\BZ)$ of divisibility $2n-2$ with $b(\gamma,\gamma)=2-2n$. We can choose the isometry \eqref{eq:choice_of_delta_for_lattive} for $X$ in such a way that $\gamma$ maps to $\delta$. With this choice the image of $\tal$ under $\varphi$ is of the form $B_{\mu}(\tal)$ for $\mu\in \h^2(S^{[n]},\BZ)_{\textup{alg}}$. 
	
	Indeed, since $\tbb(\tal,\beta) =-1$ we must have 
	\[\tbb(\varphi(\tal), \varphi(\beta)) = \tbb(\varphi(\tal),\beta) = -1
	\] 
	and similarly $\tbb(\varphi(\tal),\tde) =0$.
	Using the orthogonal decomposition
	\[
	\Lambda_{S^{[n]}} \cong \left( \BZ \tal \oplus \BZ \beta \right) \oplus^\perp \BZ \tde \oplus^\perp \h^2(S,\BZ)
	\]
	we see that $\varphi(\tal)$ is of the form 
	\[
	\varphi(\tal) = \tal + \mu + d\beta
	\]
	for some $\mu \in \h^2(S,\BZ)$ and $d\in \BZ$. As $\varphi$ is a Hodge isometry, we furthermore have $\mu \in \h^2(S,\BZ)_{\textup{alg}}$ and $b(\mu,\mu) = 2d$ which implies $\varphi(\tal) = B_\mu(\tal)$. 
	Postcomposing $\varphi$ with $B_{-\mu}$ and using \eqref{eq:ab123} we obtain a Hodge isometry
	\begin{equation}
	\label{eq:proof_cor_bir_der}
	\h^2(X,\BZ)\cong \h^2(S^{[n]},\BZ).
	\end{equation}
	
	Corollary~\ref{cor:der_equi_moduli_spaces} implies that $X$ is a moduli space $M^S_{\sigma}(v)$ of stable objects on $S$. Moreover, from \eqref{eq:proof_cor_bir_der} and the lemma further below we infer that $M^S_{\sigma}(v)$ is a fine moduli space, i.e.\ there exists $w\in \tH(S,\BZ)_{\textup{alg}}$ such that $b(v,w)=1$. Invoking again \cite[Prop.\ 3.3]{GHS_Pi1} we see that there exists 
a Hodge isometry $\gamma\in \RO^+(\tH(S,\BZ))$ such that $\gamma(v)=(1,0,1-n)$. The assertion follows now from \cite[Cor.\ 9.9]{MarkmanSurvey}.
\end{proof}
In \cite[Thm.\ B]{ADMModuli} the authors found an example of derived equivalent hyper-Kähler manifolds such that their second integral cohomology is not Hodge isometric. In particular, the above proposition does not always hold. As granted by Theorem~\ref{prop:rel_version_hodge_isom_lattice} their 
$\Kdrein$ lattices are Hodge isometric, see also Remark~\ref{rmk:explanation_ADM}.

Summarising and using \cite{HLDKConj} we have for a projective $\Kdrein$-type hyper-Kähler manifold $X$ and an elliptic K3 surface $S$ with section: $X$ and $S^{[n]}$ are birational if and only if $\h^2(X,\BZ)$ is Hodge isometric to $\h^2(S^{[n]},\BZ)$ if and only if $\Db(X)\cong \Db(S^{[n]})$. 

We finish the section with the following result used in the above proof. We will need some lattice theory and refer once more to \cite[Sec.\ 14]{HuybrechtsK3} for notations and results. Recall that a moduli space $M_\sigma^S(v)$ of stable sheaves or objects on a K3 surface $S$ is \textit{fine} if there exists a universal family $\CE$ on $M_\sigma^S(v) \times S$. This is equivalent to the existence of some $w \in \tH(S,\BZ)_{\textup{alg}}$ such that $\tbb(v,w) = 1$, see \cite[Sec.\ 10.2.2]{HuybrechtsK3}. 
\begin{lem}
	Let $M$ and $M'$ be smooth moduli spaces of stable objects on a projective K3 surface $S$ such that $\textup{NS}(M)$ and  $\textup{NS}(M')$ are isometric with respect to the BBF pairing. Then $M$ is a fine moduli space if and only if $M'$ is.
\end{lem}
\begin{proof}
	Let $N\subset L$ be a saturated sublattice of an even lattice $(L,(\_,\_))$, i.e.\ $L/N$ is torsion-free. Consider the diagram
	\begin{equation}
	\label{eq:diag_lattices}
	\begin{tikzcd}
		0 \ar[r] & N^\perp \ar[d] \ar[r] & N^\perp \ar[d] \ar[r] & 0 \ar[d] & \\
		0 \ar[r] & N \oplus N^\perp \ar[r] \ar[d] & L \ar[d] \ar[r] & K \ar[d] \ar[r] &0 \\
		0 \ar[r] & N \ar[r] \ar[d] & N^\vee \ar[d] \ar[r] & A(N) \ar[d] \ar[r] &0 \\
		& 0 \ar[r] & P \ar[r]  & P \ar[r] & 0.
	\end{tikzcd}
	\end{equation}
	Here, $N^\vee \coloneqq \Hom_\BZ(N,\BZ)$ is the dual lattice, $N\to N^\vee$ and $L \to N^\vee$ denote the natural maps $v\mapsto (x\mapsto (x,v))$, $A(N)$ is the discriminant group of $N$, and $K$ and $P$ denote the cokernel of the corresponding morphisms. 
	
	As recalled above, for a moduli space $M=M_\sigma^S(v)$ we have
	\[
	\h^2(M,\BZ) \cong v^\perp \subset \tH(S,\BZ).
	\]
	In particular, 
	\begin{equation*}
		\textup{NS}(M) \oplus \BZ v \subset \tH(S,\BZ)_\textup{alg}
	\end{equation*}
	is an orthogonal decomposition of a finite index sublattice of $\tH(S,\BZ)_\textup{alg}$. 

	Let us assume that $M = M_\sigma^S(v)$ is fine and apply diagram \eqref{eq:diag_lattices} for $L = \tH(S,\BZ)_\textup{alg}$ and $N = \BZ v$. The moduli space $M$ being fine is equivalent to surjectivity of the map
	\[
	\tH(S,\BZ)_\textup{alg} \to (\BZ v)^\vee, \quad x\mapsto (kv\mapsto \tbb(x,kv)).
	\]
	Hence, in our situation we have $P\cong 0$ and, therefore, $K\cong A(N)\cong \BZ/(2n-2)\BZ$ for $2n$ the dimension of $M$. Using \cite[Eq.\ (0.2) in Ch.\ 14]{HuybrechtsK3} we find 
	\begin{equation}
	\label{eq:9_10_a}
		\textup{disc}(\textup{NS}(M)) = |K|^2 \cdot \textup{disc}(\tH(S,\BZ)_\textup{alg}) / \textup{disc}(\BZ v) = (2n-2) \cdot \textup{disc}(\tH(S,\BZ)_\textup{alg})
	\end{equation}

	Let us now consider $M' = M_{\sigma'}^S(v')$ and inspect diagram \eqref{eq:diag_lattices} for $L = \tH(S,\BZ)_\textup{alg}$ and $N = \BZ v'$ such that $K = L/(N \oplus N^\perp)$. We employ again \cite[Eq.\ (0.2)]{HuybrechtsK3} and find
	\begin{equation}
	\label{eq:9_10_b}
		\textup{disc}(\textup{NS}(M')) = |K|^2 \cdot \textup{disc}(\tH(S,\BZ)_\textup{alg}) / (2n-2).
	\end{equation}
	By assumption, $\textup{NS}(M)$ and $\textup{NS}(M')$ are isometric, thus we obtain the equality $\textup{disc}(\textup{NS}(M)) = \textup{disc}(\textup{NS}(M'))$. Combining \eqref{eq:9_10_a} and \eqref{eq:9_10_b} we find $|K| = 2n-2$. In particular, in the situation $L = \tH(S,\BZ)_\textup{alg}$ and $N = \BZ v'$ we find that $K \cong A(N)$ and, therefore, $P \cong 0$ in \eqref{eq:diag_lattices}. This implies that
	\[
	\tH(S,\BZ)_\textup{alg} \to (\BZ v')^\vee, x\mapsto (kv'\mapsto \tbb(x,kv'))
	\]
	is surjective which shows that $M'$ is a fine moduli space as well. 
\end{proof}
We remark that the proof also applies for non-fine moduli spaces $M$ and $M'$. That is, in general, there always exists Brauer classes $\alpha, \alpha' \in \mathrm{Br}(S)$ such that $\alpha$ respectively $\alpha'$ twisted universal families exist over $M \times S$ respectively $M' \times S$. The proof then shows that $\textup{ord}(\alpha) = \textup{ord}(\alpha')$ if $\textup{NS}(M) \cong \textup{NS}(M')$. 
\section{Further examples of derived equivalences}
\label{sec:Calculation_exampleequiv_on_extended}
We complement the previous sections by integrating some derived equivalences of hyper-Kähler manifolds into the framework of the extended Mukai lattice.  
\subsection{Dimension four}
\label{subsec:examples_dim4}
We come back to Example~\ref{ex:4.15}. Addington \cite{AddingtonDerSymHK} 
as well as Markman--Mehrotra \cite{MarkmanMehrotraIntTransf} considered the sheaf
\[
\mathcal{E}^1\coloneqq \sExtA^1_{\pi_{13}}(\pi_{12}^*(\mathcal{I}), \pi_{23}^*(\mathcal{I})) \in \Coh(S^{[2]} \times S^{[2]})
\]
which is reflexive, of rank 2 and locally free away from the diagonal. Here, $\CI$ is the universal ideal sheaf on $S \times S^{[2]}$ and $\pi_{ij}$ are the projections from $S^{[2]}\times S \times S^{[2]}$. The Fourier--Mukai transform $\FM_{\CE^1}$ with kernel $\CE^1$ was shown to yield an 
auto-equivalence $\FM_{\CE^1}\in \Aut(\Db(S^{[2]}))$. More conceptually, the functor $\FM_{\CI}$ is shown to be a spherical functor with $\CE^1[1]$ the corresponding twist auto-equivalence. 

\begin{prop}
\label{prop:Markman_Ext_on_Mukailattice}
	The equivalence $\FM_{\CE^1}$ acts on the extended Mukai lattice via $-s_v$, where $v$ is the vector $\tal + \beta$.
\end{prop}
\begin{proof}
	One way to prove the assertion is to use general results on the action of the twist equivalence associated to a spherical functor \cite[Sec.\ 1.4]{AddingtonDerSymHK}.
	Instead, we will calculate directly the images of line bundles using the 
definition of the twist auto-equivalence associated to a spherical functor. 
	
	The relative $\Ext$ complex
	\[
	\CE\coloneqq\sExtA_{\pi_{13}}^{\bullet}(\pi_{12}^*(\mathcal{I}), \pi_{23}^*(\mathcal{I}))={\pi_{13}}_\ast (\pi_{12}^{\ast}(\CI^{\vee})\otimes \pi_{23}^{\ast}(\CI))\in \Db(S^{[2]} \times S^{[2]})
	\]
	describes (up to the shift $[2]$) the composition of the right adjoint of 
	\[
	\FM_\CI \colon \Db(S^{[2]}) \to \Db(S)
	\] 
	with $\FM_\CI$ and sits in a distinguished triangle
	\[
	\CE^1[-1]\to \CE\to \CO_{\Delta}[-2]
	\]
	in $\Db(S^{[2]}\times S^{[2]})$. This yields the identity
	\[
	[\FM_{\CE^1}(\CL_2)]=[\CL_2]-[\FM_{\CE}(\CL_2)] = [\CL_2] - [\FM_\CI(\FM_{\CI^{\vee}}(\CL_2))]
	\]
	in topological $K$-theory for all topological line bundles $\CL$ on $S$. 

	There is a natural short exact sequence
	\[
	0 \to \CI \to \CO_{S\times S^{[2]}} \to \CO_{\CZ} \to 0
	\]
	on $S\times S^{[2]}$, where $\CZ \subset S\times S^{[2]}$ is the universal subscheme, which we can dualize to obtain the distinguished triangle
	\[
	\omega_{\CZ} [-2] \to \CO_{S\times S^{[2]} } \to \CI^{\vee}
	\]
	in $\Db(S\times S^{[2]})$. From these sequences we obtain the identities
	\begin{align*}
	[\FM_{\CI^{\vee}}(\CL_2)]&=\chi(\CL_2)[\CO_S]-\chi(\CL)[\CL],\\
	[\FM_{\CI}(\CO_S)] & = 2[\CO_{S^{[2]}}]-[\CO_{S^{[2]}}]-[\CO_{S^{[2]}}(-\delta)],\\
	[\FM_{\CI}(\CL)]&=\chi(\CL)[\CO_{S^{[2]}}] - [\CL^{[2]}]
	\end{align*}
	in topological $K$-theory, where $\CL^{[2]}=\FM_{\CO_\CZ}(\CL)$ is the 
tautological rank 2 bundle associated to $\CL$ and the second identity is a special case of the third one using $\CO_S^{[2]} \cong \CO_{S^{[2]}} \oplus \CO_{S^{[2]}}(-\delta)$. The class of a point $\pt$ is sent to a sheaf of rank 2. By an analogous argument to Lemma~\ref{lem:determinant_extended_rank} using the object $k(x)$ one concludes $\epsilon(\FM_{\CE^1}^{\tH})=1$. 

To finish the proof we need to use the above to calculate how $\FM_{\CE^1}$ acts on the extended Mukai lattice $\tH(S^{[2]},\BQ)$. We have $[\FM_{\CE^1}(\CO_{S^{[2]}})] = [\CO_{S^{[2]}}(-\delta)]$ as well as $[\FM_{\CE^1}(\CO_{S^{[2]}}(-\delta))] = [\CO_{S^{[2]}}]$, since spherical functors always induce involutions on cohomology. Applying extended Mukai vectors to this equality we find $\FM_{\CE^1}^{\tH}(\tilde{v}(\CO_{S^{[2]}})) = \tilde{v} (\CO_{S^{[2]}}(-\delta))$ and vice versa. 

For a general topological line bundle $\CL$ we find
\[
[\FM_{\CE^1}(\CL_2)] = [\CL_2] - \chi(\CL_2)([\CO_{S^{[2]}}]-[\CO_{S^{[2]}}(-\delta)]) - \chi(\CL)^2[\CO_{S^{[2]}}] + \chi(\CL) [\CL^{[2]}].
\]
If $b(\Rc_1(\CL),\Rc_1(\CL)) = \ell$, then $\chi(\CL) = \ell/2+2$ and $\chi(\CL_2) = \ell^2/8 + 5\ell/2 +3$, see \cite[Lem.\ 5.1]{EGL}. Moreover, $\CL^{[2]}$ is a bundle of rank two and $\Rc_1(\CL^{[2]}) = \Rc_1(\CL) -\delta$. Taking extended Mukai vectors an explicit calculation shows that $\FM_{\CE^1}^{\tH}$ agrees with $-s_v$ for $v = \tal + \beta$. 
\end{proof}
Thus, the functors $\FM_{\CE^1}$ and $\phi_{[2]}(\ST_{\CO_S})$ induce the 
same isometry on the extended Mukai lattice and therefore also on the whole cohomology.

For $S^{[2]}$ there are other auto-equivalences given as the twist of a spherical functor. One example is Horja's EZ-spherical twist \cite{HorjaEZ}. The exceptional divisor $i\colon\BP(\Omega_S^1) \cong E \hookrightarrow S^{[2]}$ fibres over the K3 surface $\pi \colon E \to S$. One obtains the spherical functor $i_{\ast}(\pi^{\ast}(\_))\colon \Db(S)\to \Db(S^{[2]})$ and an auto-equivalence $T_{i_{\ast}\pi^{\ast}}\in \Aut(\Db(S^{[2]}))$ characterized for $\CF\in \Db(S^{[2]})$ by the distinguished triangle
\[
i_{\ast}\pi^{\ast}\pi_{\ast}i^{!}(\CF)\to \CF\to T_{i_{\ast}\pi^{\ast}}(\CF).
\]
\begin{prop}
\label{prop:Horja_EZ_on_Mukai}
	The auto-equivalence $T_{i_{\ast}\pi^{\ast}}$ acts on the extended Mukai 
lattice via the isometry $-s_{v}$ for the vector $v=\tde + \beta$.
\end{prop}
\begin{proof}
	We employ \cite[Sec.\ 2.4]{AddingtonDerSymHK}. Lemma~\ref{lem:determinant_extended_rank} gives once more $\epsilon(T_{i_{\ast}\pi^{\ast}}^{\tH})=1$. For $\CL$ a line bundle on $S$ one easily obtains
	\[
	i_\ast\pi^\ast(\CL^{\otimes 2})\cong \CL_2|_E. 
	\]
	This means that the classes $[\CL_2]-[\CL_2\otimes \CO_{S^{[2]}}(-E)]$ in $K^0_{\textup{top}}(S^{[2]})$ are being multiplied by $-1$ under the action of $T_{i_{\ast}\pi^{\ast}}$ and their orthogonal complement is left invariant. We have an equality
	\[
	2(v(\CL_2)-v(\CL_2\otimes \CO_{S^{[2]}}(-E))=T(\tilde{v}(\CL_2)^2-\tilde{v}(\CL_2\otimes \CO_{S^{[2]}}(-E))^2)
	\]
	in $\SH(S^{[2]},\BQ)$. A computation finishes the proof.
\end{proof}
Alternatively, one could have proven the proposition using \cite[Thm.\ 4.26]{KrugPloogSosna} and Proposition~\ref{lem:signrep_on_coh}. 
\begin{rmk}
	All cohomological involutions we encountered (Proposition~\ref{lem:signrep_on_coh}, Pro-position~\ref{prop:STonHilbn}, Proposition~\ref{prop:Markman_Ext_on_Mukailattice} and Proposition~\ref{prop:Horja_EZ_on_Mukai}) had an extra $-\id$ in the case $n$ even which is due to the sign convention from Section~\ref{subsec:prelim_der_equi_tael}. 
\end{rmk}
\subsection{Relative Poincar\'e}
\label{subsec:Example_POINCARE_ADM}
In \cite{ADMModuli} the authors study derived equivalences between certain moduli spaces of stable sheaves on K3 surfaces. Let $S$ be a very general projective K3 surface with polarization $H$ of degree $2g-2$ and consider the moduli spaces of stable sheaves
\[
M^S_H(0,1,d+1-g).
\]
One can equivalently consider these varieties as the relative compactified Jacobians $\overline{\Pic}^d\coloneqq \overline{\Pic}^d(\CC/\BP^g)$ of degree $d$ of the universal curve
\[
\CC \to \BP^{g}=|H|.
\] 
In \cite[Prop.\ 3.1]{ADMModuli} following Arinkin \cite{ArinkinAutoduality} the authors construct a relative (twisted) Poincar\'e sheaf $\CP_{dd'}$ on 
\[
\overline{\Pic}^d \times_{\BP^g} \overline{\Pic}^{d'}
\]
inducing a (twisted) derived equivalence. For simplicity we will consider 
the untwisted case $d=d'=0$ and denote $\CP\coloneqq\CP_{00}$, $M\coloneqq M^S_H(0,1,1-g)$ together with the Lagrangian fibration
\[
\pi \colon M \to \BP^g.
\] We have the well-known Hodge isometry
\[
\h^2(M,\BZ)\cong (0,1,1-g)^\perp \subset \tH(S,\BZ)
\]
and the algebraic part of $\h^2(M,\BZ)$ has a basis $\lambda,f$ with intersection form
\[
\begin{pmatrix} 2g-2 & 2\\ 2 & 0 \end{pmatrix}
\]
where $f$ is the first Chern class of $\CO_{\BP^g}(1)$ pulled back to $M$. 

Let us determine the action of $\CP$ on the extended Mukai lattice. We will consider the case $g$ even, the case $g$ even is similar. The skyscraper sheaf $k(x)$ of a point 
$x\in A\subset M$ contained in a smooth fibre $A$ of the Lagrangian fibration is by definition sent under $\FM_\CP$ to a degree 0 line bundle $\CL$ on the abelian variety $A$ whose Mukai vector is of the form $v(\CL)=f^g \in \SH(X,\BQ)$. The duality property of the Poincar\'e sheaf \cite[Sec.\ 6.2]{ArinkinAutoduality} implies that $\CL$ is sent under $\FM_\CP$ to the object $k(x^{\vee})[-g]$, where $x^{\vee}\in A$ parametrizes $\CL^{\vee}$. This gives
\[
\beta \mapsto f, \quad f \mapsto \beta.
\]

Moreover, the Lagrangian fibration $M\to \BP^g$ admits a section $\BP^g \hookrightarrow M$ given by the trivial line bundle on each fibre. Using Remark~\ref{rem:Pn_in_Hilb} and $\int_X[A][\BP^g]=1$ we see that the Mukai vector of $\CO_{\BP^g}\in \Db(M)$ satisfies
\[
\overline{v(\CO_{\BP^g})}=T\left( \frac{(\frac{1}{2}\lambda-\frac{g+1}{2}f + \frac{g+1}{2}\beta)^g}{g!} \right)\in \SH(M,\BQ).
\]
The definition of $\CP$ \cite[Eq.\ (3.1)]{ADMModuli} yields that $\FM_\CP$ sends $\CO_{\BP^g}$ to a line bundle $\CM\in \Pic(M)$. The duality property of $\CP$ for families of curves \cite[Eq.\ (7.8)]{ArinkinAutoduality} implies that $\CM$ is mapped under $\FM_\CP$ to $\CO_{\BP^g}[-g]\otimes 
\CK$. Here, $\CK$ is the line bundle $\pi^{\ast}\det(\mathrm{R}^1\pi_{\ast}\CO_M)$ which, using \cite[Thm.\ 1.3]{MatsushitaHigherImages}, has first Chern class $-(g+1)f$. 
Let us denote
\[
h=-\frac{\lambda}{2}+\frac{g-1}{4}f +\frac{g+1}{2}\beta \in \tH(M,\BQ)
\]
and note that $f$ and $-h$ span a rational hyperbolic plane. Summarizing the above discussion and using the extended Mukai vector we have the following.
\begin{prop}
	The equivalence $\FM_\CP$ acts on the extended Mukai lattice via
	\[
	\alpha \mapsto h, \quad h\mapsto \alpha, \quad \beta \mapsto f, \quad f\mapsto \beta.
	\]
\end{prop}
Expressed differently, the derived equivalence $\FM_\CP$ exchanges the two rational hyperbolic planes given by $\alpha, \beta$ and $f,h$. 

\begin{rmk}
\label{rmk:explanation_ADM}
In \cite{ADMModuli} it was observed that the case $d=0$ and $d'=g-1$ yields an example of derived equivalent hyper-Kähler manifolds $M$ and $M'$ such that their second integral cohomology groups are not isometric \cite[Thm.\ B]{ADMModuli}. The intersection form on $\textup{NS}(M)$ has discriminant $-4$ whereas the lattice $\textup{NS}(M')$ is isometric to the hyperbolic plane. Let us denote the generators of $\textup{NS}(M')$ inside $\h^2(M',\BZ)$ by $e',f'$ such that $b(e',f')=1$ and $f'$ denotes again 
the fibre class. As above one can show that the derived equivalence induces an isometry $\tH(M,\BQ)\cong  \tH(M',\BQ)$ given by
\[
\beta\mapsto f', \quad f\mapsto \beta, \quad \alpha \mapsto - e'+\frac{g+1}{2}\beta, \quad h \mapsto \alpha.
\]
This is compatible with the $\Kdrein$ lattices, i.e.\ the above induces a 
Hodge isometry
\[
\Lambda_M\cong \Lambda_{M'}
\]
in accordance with Theorem~\ref{prop:rel_version_hodge_isom_lattice}. Geometrically the variety $M$ admits a section whereas the variety $M'$ admits a line bundle with first Chern class $e'$ which restricts to a principal polarization on each fibre. The derived equivalence $\CP_{0g-1}$ relates these different geometric properties.
\end{rmk}
	\bibliography{../pub_bib}{}
	\Addresses
\end{document}